\documentclass{amsart}

\usepackage{mathptmx}
\usepackage{amssymb}
\usepackage[margin=2.5cm]{geometry}   % remove before submission
\usepackage{graphicx}
\graphicspath{{./images/}}
\usepackage{overpic}
\usepackage{float}
\usepackage{array}
\usepackage{enumerate}
\usepackage{systeme}
\usepackage[all]{xy}
\usepackage{xcolor}
\usepackage{amsaddr}
\usepackage[hidelinks]{hyperref}      % must stay last

\newtheorem{theorem}{Theorem}[section]
\newtheorem{proposition}[theorem]{Proposition}

\newtheorem{lemma}{Lemma}
\newtheorem{definition}{Definition}
\newtheorem{remark}{Remark}

\newtheorem{ithm}{Theorem}

\def\cal#1{\mathcal{#1}}
\def\bb#1{\mathbb{#1}}
\def\lie#1{\mathfrak{#1}}
\def\lr#1{\left\langle #1\right\rangle}

\newcommand{\ga}{\textsl{g}}
\newcommand{\aga}{\textsl{h}}
\newcommand{\ov}[1]{\overline{#1}}

\newcommand{\h}{\tfrac{1}{2}}

\newcommand{\Rsec}[1]{\mathcal{R}^{\text{sec-Ric}}_{#1}}
\newcommand{\Rric}[1]{\mathcal{R}^{\text{Ric-scal}}_{#1}}
\newcommand{\Ricd}{\Rsec{d}}

\newcommand{\ricd}{\Rric{d}}

\DeclareMathOperator{\Ricci}{Ric}

\DeclareMathOperator{\Ad}{Ad}

\newcommand{\R}{\mathbb{R}}

\newcommand{\CP}{\mathbb{C\mkern1mu P}}

\title{The complete dynamics description of positively curved metrics in the
Wallach flag manifold $\mathrm{SU}(3)/\mathrm{T}^2$}

\author[Cavenaghi]{Leonardo F. Cavenaghi}
\address{Institute of Mathematics and Informatics, Bulgarian Academy of Sciences}
\email{leonardofcavenaghi@gmail.com}

\author[Grama]{Lino Grama}
\address{Instituto de Matem\'atica, Estat\'istica e Computa\c{c}\~ao Cient\'ifica --
Unicamp, Rua S\'ergio Buarque de Holanda, 651, 13083-859, Campinas, SP, Brazil}
\email{lgrama@unicamp.br}

\author[Martins]{Ricardo M. Martins}
\address{Instituto de Matem\'atica, Estat\'istica e Computa\c{c}\~ao Cient\'ifica --
Unicamp, Rua S\'ergio Buarque de Holanda, 651, 13083-859, Campinas, SP, Brazil}
\email{rmiranda@unicamp.br}

\author[Novaes]{Douglas D. Novaes}
\address{Instituto de Matem\'atica, Estat\'istica e Computa\c{c}\~ao Cient\'ifica --
Unicamp, Rua S\'ergio Buarque de Holanda, 651, 13083-859, Campinas, SP, Brazil}
\email{ddnovaes@unicamp.br}


\begin{document}
\begin{abstract}
The family of invariant Riemannian manifolds in the Wallach flag manifold $\mathrm{SU}(3)/\mathrm{T}^2$ is described by three parameters $(x,y,z)$ of positive real numbers. By restricting such a family of metrics in the \emph{tetrahedron} $\cal{T}:= x+y+z = 1$, in this paper, we describe all regions $\cal R \subset \cal T$ admitting metrics with curvature properties varying from positive sectional curvature to positive scalar curvature, including positive intermediate curvature notion's. We study the dynamics of such regions under the \emph{projected Ricci flow} in the plane $(x,y)$, concluding sign curvature maintenance and escaping. In addition, we obtain some results for positive intermediate Ricci curvature for a path of metrics on fiber bundles over $\mathrm{SU}(3)/\mathrm{T}^2$, further studying its evolution under the Ricci flow on the base.
\end{abstract}

\maketitle

\section{Introduction}

Since the paper of B\"ohm and Wilking \cite{BW}, which uses Ricci flow techniques to prove a sort of converse to the classical Bonnet--Meyers theorem, a folkloric aspect emerged on the dynamics of invariant metrics in the Wallach flag manifold $\mathrm{SU}(3)/\mathrm{T}^2$. To know, on the one hand, Theorem C in \cite{BW} deals with the compact manifold $M = \mathrm{Sp}(3)/\mathrm{Sp}(1)\times \mathrm{Sp}(1)\times \mathrm{Sp}(1)$, ensuring the existence of an invariant metric on such a homogeneous space which has positive sectional curvature and evolves under the so-called homogeneous Ricci flow in a metric with mixed Ricci curvature. %The invariant metric of $\mathrm{Sp}(3)$ induces on $M$ a homogeneous unit volume Einstein metric $\ga_E$ of non-negative sectional curvature. Associated to such homogeneous space, one can consider the fibration $\mathrm{S}^4\hookrightarrow M\rightarrow \bb{H}P^2$ and combining a \emph{canonical variation} on the Riemannian submersion metric obtained out $\ga_E$ plus scaling; resulting in a curve $\ga_t,~t>1$ of unit volume submersion metrics with positive sectional curvature for which $\ga_0 = \ga_E$. Up to parametrization, such a curve is a solution to the normalized Ricci flow. A precise analysis of the asymptotic behavior of solutions of the Ricci flow allows the authors to prove that for any homogeneous non-submersion initial metric, being close enough to $\ga_2$, the normalized Ricci flow evolves mixed Ricci curvature. 
On the other hand, however, the very prolific analysis in B\"ohm--Wilking's paper cannot be straightforwardly applied to the manifold $\mathrm{SU}(3)/\mathrm{T}^2$, and Remark 3.2 in \cite{BW} states the existence of an invariant metric with positive Ricci curvature on the flag manifold $\mathrm{SU}(3)/\mathrm{T}^2$ that evolves under the homogeneous Ricci flow to a metric with mixed Ricci curvature. Some works appeared later, seeking to give different descriptions for such a metric evolution: \cite{wallachrelated, Abiev2016}. We also observe that the study of geometric flows of invariant geometric structures on homogeneous spaces and Lie groups is a classic topic in Differential Geometry with recent developments. See, for instance, \cite{MR3653239,MR4151344,MR3957836,MR3964154,MR4105515,MR4244880,elauret} and references therein.

In this paper, via a different tool, we provide a complete description of each invariant positively curved metric in $\mathrm{SU}(3)/\mathrm{T}^2$ for every notion of positive curvature interpolating between positive sectional curvature to positive scalar curvature, further studying the dynamic evolution of such metrics under a projected homogeneous Ricci flow. The approach to this problem via dynamical systems techniques has been very successful, see for example \cite{GRAMA2009463}.

Theorem \ref{ithm:main} below fully generalizes Theorem 1 in \cite{Abiev2016} for $\mathrm{SU(3)}/\mathrm{T}^2$ (there denoted by $W^6$), further extending Theorem 2 in the same reference for $W^6$, strengthening the results for all intermediate positive curvature notations. It also provides a complete description of Theorem 3 in \cite{Abiev2016} and fully generalizes \cite{cavenaghi2023dynamics}. In Theorem \ref{ithm:main}, to interpolate between positive sectional curvature and positive Ricci curvature, we use the following concepts appearing in the literature, observing, however, that for the forthcoming definitions, there is no widely used notation/terminology. 

\begin{definition}
Given a point $p$ in a Riemannian manifold $(M,\ga)$, and a collection $v,v_1,\ldots,v_{d}$ of orthonormal vectors in $T_pM$, the $d^{\mathrm{th}}$-intermediate Ricci curvature at $p$ corresponding to this choice of vectors is defined to be $\Ricci_d(v) = \sum_{i=1}^{d} K(v,v_i),$ where $K$ denotes the sectional curvature of $\ga$.
\end{definition}

\begin{definition}[$d\mathrm{th}$-intermediate positive Ricci curvature]\label{def:lm}
    We say that a Riemannian manifold $(M,\ga)$ has positive $d\mathrm{th}$-Ricci curvature if for every $p\in M$ and every choice of non-zero $d+1$-vectors $\{v,v_1,\ldots,v_d\}$ where $\{v_1,\ldots,v_d\}$ can be completed to generate an orthonormal frame in $T_pM$, it holds $\Ricci_d(v) > 0$. 
\end{definition}

It is remarkable that for an $n$-dimensional manifold, these curvatures interpolate between positive sectional curvature and positive Ricci curvature for $d$ ranging between $1$ and $n-1$. Quoting \cite{lm}, the quantity presented in Definition \ref{def:lm} has been called ``$d\mathrm{th}$-intermediate Ricci curvature'', ``$d\mathrm{th}$-Ricci curvature'', ``$d$-dimensional partial Ricci curvature'', and ``$d$-mean curvature''. Considering this, we define
\begin{definition}\label{defi:sec-ric}
    Let $M^n$ be a $n$-dimensional manifold and fix $d\in \{1,\ldots,n-1\}$. We denote by $\Ricd$ the set of all Riemannian metrics in $M^n$ which satisfies Definition \ref{def:lm}, that is, that have $d\mathrm{th}$-intermediate positive Ricci curvature.
\end{definition}

All invariant Riemannian metrics in the Wallach flag manifold $\mathrm{SU}(3)/\mathrm{T}^2$ can be described by three positive parameters. Hence, we can abuse the notation and denote an invariant metric $\ga$ in $\mathrm{SU}(3)/\mathrm{T}^2$ by $\ga = (x,y,z)$. On the other hand, we will always assume that $x+y+z = 1$, so we only have two-parameter describing any invariant Riemannian metric, thus adopting the convention $\ga = (x,y,z) = (x,y,1-x-y) \equiv (x,y)$. We prove:

%\begin{ithm}\label{ithm:main}
%Let $\mathrm{SU}(3)/\mathrm{T}^2$ be the $6$-dimensional Wallach flag manifold. Then for any $d\in \{1,\ldots,5\}$ the set $\Ricd$ is non-empty. Moreover, fixed $d \in \{1,2,3,5\}$, there exists an invariant Riemannian metric $\ga = (x_0,y_0) \in \Ricd$ such that:\begin{enumerate}[(a)]\item the projected homogeneous Ricci flow $\ga(t) = (x(t),y(t))$ with initial condition $\ga(0) = \ga$ belongs to $\Ricd$ for $t$-sufficiently small\item there exists $t^* \in \bb R$ such that $\ga(t) \not\in \Ricd$ for every $t$ in a neighborhood of $t^*$\item \label{item:invariance} fixed $d\in \{4,5\}$ there are regions $\cal R_d\subseteq \Ricd$ such that for any invariant Riemannian metric $\ga = (x_0,y_0) \in \cal R_d$, the projected homogeneous Ricci flow $\ga(t)$ with initial condition $\ga(0) = \ga$ lies in $\cal R_d$ for every $t\in \bb R$. Moreover, for $d = 4$ we have $\cal R_4 = \mathcal{R}^{\rm{\, sec-Ric}}_4$ (which is open) and the corresponding region $\cal R_5$ for $d = 5$ is the closure 
%    \begin{equation}
%     \overline{\mathcal{R}^{\rm{\, sec-Ric}}_{4}} = \cal R_5
%    \end{equation}
%\end{enumerate}
%\end{ithm}

\begin{ithm}\label{ithm:main}
Let $\mathrm{SU}(3)/\mathrm{T}^2$ be the $6$-dimensional Wallach flag manifold. Then for any $d\in \{1,\ldots,5\}$ the set $\Ricd$ is non-empty. Moreover, 
\begin{enumerate}[(a)]
    \item for each $d \in \{1,2,3,5\}$, there exists an invariant Riemannian metric $\ga = (x_0,y_0) \in \Ricd$ and $t^* \in \bb R$ such that the projected homogeneous Ricci flow $\ga(t) = (x(t),y(t))$ with initial condition $\ga(0) = \ga$ belongs to $\Ricd$ for $t$-sufficiently small and $\ga(t) \not\in \Ricd$ for every $t$ in a neighborhood of $t^*$;
    \item \label{item:invariance} for each $d\in \{4,5\}$, there exists a region $\cal R_d\subseteq \Ricd$ such that for any invariant Riemannian metric $\ga = (x_0,y_0) \in \cal R_d$, the projected homogeneous Ricci flow $\ga(t)$ with initial condition $\ga(0) = \ga$ lies in $\cal R_d$ for every $t\in \bb R$. Moreover,  $\cal R_4 = \mathcal{R}^{\rm{\, sec-Ric}}_4$ (which is open) and $\cal R_5=\overline{\cal R_4}\setminus\{(1/2,0),(0,1/2),(1/2,1/2)\}\subsetneq \mathcal{R}^{\rm{\, sec-Ric}}_5$.
\end{enumerate}
\end{ithm}

The projected Ricci flow's global behavior and the sets' dynamics $\Ricd$ under the flow described in Theorem \ref{ithm:main} is illustrated in Figure \ref{fig:sec-ric}. 
 
Another related notion considered here of \emph{intermediate curvature condition} is
\begin{definition}\label{def:minimun}
Let $M^n$ be a $n$-dimensional Riemannian manifold and let $d\leq n$. We say that the Ricci tensor of $M$ is \emph{$d$-positive} if the sum of the $d$ smallest eigenvalues of the Ricci tensor is positive at all points.
\end{definition}

It is worth pointing out that if $d$ ranges from $1,\ldots,n$, the condition given by Definition \ref{def:minimun} interpolates between positive Ricci curvature and positive scalar curvature. Definition \ref{def:ours} below, to be further approached in Section \ref{sec:def2}, considers a generalization of Definition \ref{def:minimun}. We do not require that the Ricci tensor is $d$-positive. Instead, we look for positive combinations of the Ricci eigenvalues constrained by their multiplicity. Such a consideration encompasses Definition \ref{def:minimun}.
\begin{definition}\label{def:ours}
    Let $M^n$ be a $n$-dimensional Riemannian manifold. Fix $d\in \{1,\ldots,n\}$. We say that a Riemannian metric $\ga$ on $M$ has positive $d$-curvature if, denoting by $\lambda_1,\ldots,\lambda_l$ the distinct eigenvalues of the Ricci tensor $\Ricci(\ga)$ of $\ga$, for every collection of non-negative integers $\{(a_1,\ldots,a_l) \in \bb N^l : a_1 +\ldots+a_l = d,~ 0\leq a_j\leq \mu(\lambda_j)\}$ where $\mu(\lambda_j)$ denotes the algebraic multiplicity of $\lambda_j$,
    \[\sum_{j=1}^la_j\lambda_j > 0.\]
\end{definition}
\begin{definition}\label{defi:ric-scal}
    Let $M^n$ be a $n$-dimensional manifold. Fix $d\in \{1,\ldots,n\}$. We denote the set of all Riemannian metrics $\ga$ on $M$ satisfying Definition \ref{def:ours} by $\ricd$.
\end{definition}

\begin{remark}
    Caution is required, since $d$-positivity of the Ricci tensor (Definition \ref{def:minimun}) is denoted by $\Ricci_d>0$ in \cite{crowley2020intermediate}. See \cite[Section 2.2]{DGM} or \cite[p. 5]{crowley2020intermediate} for further information.
\end{remark}

We prove:
\begin{ithm}\label{ithm:main3}
    Let $\mathrm{SU}(3)/\mathrm{T}^2$ be the $6$-dimensional Wallach flag manifold. Then for any $d\in \{1,\ldots,6\}$ the set $\ricd$ is non-empty. Moreover, for each $d \in \{1,\ldots,6\}$, 
\begin{enumerate}[(a)]
\item there exists an invariant Riemannian metric $\ga = (x_0,y_0) \in \ricd$ and $t^*\in  \R$ such the projected homogeneous Ricci flow $\ga(t) = (x(t),y(t))$ with initial condition $\ga(0) = \ga$ belongs to $\ricd$ for $t$ in a neighborhood of $0$, whereas $\ga(t)\notin\ricd$ for $t$ in a neighborhood of $t^*$. In particular, the flow crosses the boundary of $\ricd$ in forward time.
   %\item there exists an invariant Riemannian metric $\ga = (x_0,y_0) \in \ricd$ and $t^*\in  \R$ such the projected homogeneous Ricci flow $\ga(t) = (x(t),y(t))$ with initial condition $\ga(0) = \ga$ belongs to $\ricd$ for $t$-sufficiently small
    and $\ga(t) \not\in \ricd$ for every $t$ in a neighborhood of $t^*$;
    
    \item there exists an open region $\cal R\subsetneq \ricd$ such that for any $\ga = (x_0,y_0) \in \cal R$, the homogeneous Ricci flow $\ga(t)$ with initial condition $\ga(0) = \ga$ lies in $\cal R_d$ for every $t\in \bb R$.
\end{enumerate}
\end{ithm}

We emphasize that the behavior described in Theorem~\ref{ithm:main3} has to be understood with a distinction between forward and backward time. Indeed, the scalar curvature is nondecreasing along the Ricci flow. Thus, although the flow may enter and leave $\ricd$, these two phenomena are not symmetric with respect to time reversal. In particular, the direction of time in which the trajectory enters or exits $\ricd$ is relevant when describing its behavior.

\begin{remark}  
An interesting consequence of the case $d=6$ of Theorem~B is that there is no ancient solution outside $\overline{\mathcal{T}_c^\circ}$, since by Proposition~4.1 in \cite{Lafuente2015}, an ancient homogeneous solution must have $S(g(t)) \geq 0$ for all $t$.
\end{remark}

The projected Ricci flow's global behavior and the sets' dynamics $\Ricd$ under the flow described in Theorem \ref{ithm:main3} is illustrated in Figure \ref{fig:ric-scal}. 

Observe that Theorems \ref{ithm:main} and \ref{ithm:main3}, when considered in perspective, make us wonder whether we could lift positive intermediate curvature notions to fiber bundles, as in \cite{reiser2022positive}. Pursuing positive answers, the following could also be seen as analogous to Theorem A in \cite{searle2015lift}.

Given a maximal closed subgroup $H<\mathrm{SU}(3)$, we write $m := \dim\mathrm{SU}(3)/H\in\{4,5,6,8\}$ and let $\delta = \delta(H)\in\{1,3,2,4\}$ denote the minimal $d$ for which the normal homogeneous metric on $\mathrm{SU}(3)/H$ has positive $\Ricci_d$ (see Lemma \ref{lem:propersubgroups}).

\begin{ithm}\label{thm:associated}
Let $P$ be a principal $\mathrm{SU}(3)$-bundle over the Wallach flag manifold $B = \mathrm{SU}(3)/\mathrm T^2$, let $\mathrm{SU}(3)$ act smoothly on a compact connected manifold $F$, and let $\pi\colon M := P\times_{\mathrm{SU}(3)}F\to B$ be the associated fiber bundle. Assume that
\begin{enumerate}[(a)]
\item a principal orbit of the $\mathrm{SU}(3)$-action on $F$ has as isotropy subgroup a maximal closed subgroup $H<\mathrm{SU}(3)$, and the fixed point set of the action is finite;
\item for some fixed $1\leq d_F\leq \ell-1$, where $\ell := \dim F-m$, the fiber $F$ admits an $\mathrm{SU}(3)$-invariant metric $\ga_F$ whose induced orbit distance metric on $F^{reg}/\mathrm{SU}(3)$ has positive $\Ricci_{d_F}$.
\end{enumerate}
Set
\[k_2 := \max\{d_F+m,\ \ell+\delta\},\qquad k_0 := k_2+6,\qquad n_B := k_0-\dim F,\]
so that $2\leq n_B\leq 5$ and $k_0\leq\dim M-1$. Then, for every invariant metric $\ga_B\in\Rsec{n_B}$ on $B$, there exist $\tau>0$ and a one-parameter family $\ga_t$, $t\in\bb R$, of Riemannian submersion metrics on $M$, constructed over the solution $\ga_B(t)$ of the projected homogeneous Ricci flow with $\ga_B(0)=\ga_B$, such that
\[\ga_t\in\Rsec{k}\quad\text{for every } k\geq k_0 \text{ and every } |t|\leq\tau.\]
Moreover, if $n_B\neq 4$ and $\ga_B$ is chosen as in Theorem \ref{ithm:main}(a) for $d = n_B${, namely,} $\ga_B\in\Rsec{n_B}\setminus(\overline{UL}\cup\overline{UM}\cup\overline{UN})$ for $n_B\in\{2,3\}$, and $\ga_B\in\Rsec{5}\setminus\overline{\cal T_c}$ for $n_B = 5$, then there exists $t^*\in\bb R$ such that $\ga_t\notin\Rsec{k_0}$ for every $t$ in a neighborhood of $t^*$. For $n_B = 4$ no such escaping occurs.
\end{ithm}

Much work related to Definition \ref{def:lm} has been appearing, and recent attention to this subject can be noticed. For a complete list of references on the subject, we recommend \cite{lm}. Part of the idea of this paper was conceived looking at the examples built in \cite{DGM} and explored in \cite{davide}, studying the evolution of intermediate positive Ricci curvature (in the sense of Definition \ref{def:lm}) on some \emph{generalized Wallach spaces}, under the homogeneous Ricci flow. Their analyses closely follow the techniques developed in \cite{BW}. Of particular interest are also \cite{kennard2023positive, davide, reiser2022positive, reiser2022intermediate, Mouille2022}.

\section{The curvature formulae in \texorpdfstring{$\mathrm{SU}(3)/\mathrm{T}^2$}{SU(3)/T2}}

 We first give an explicit \emph{Weyl basis} to $T_o\mathrm{SU}(3)/\mathrm{T}^2$ where $o = e\mathrm{T}^2$, being $e\in \mathrm{SU}(3)$ the unit element. More precisely, from the level of Lie algebra, let $\mathfrak{sl}(3;\mathbb{C})$ be the complex semisimple Lie algebra with compact real form $\lie{su(3)}$, the Lie algebra of $\mathrm{SU}(3)$. It can be checked that $\mathfrak{t}^2$ is a \emph{Cartan sub-algebra} of $\lie{sl}(3;\bb C)$, where $\lie{t}^2 =: \mathrm{Lie}(\mathrm{T}^2) = T_e\mathrm{T}^2$.
 
Since $\lie{t}^2$ is Abelian and the adjoint representation $\mathrm{ad}: \lie{sl}(3;\bb C) \rightarrow \lie{gl}(\lie{sl}(3;\bb C))$ enjoys the property that the image $\mathrm{ad}(\lie{t}^2)$ consists of \emph{semisimple operators}, therefore it can be simultaneously diagonalizable. Considering this, we thus decompose $\lie{sl}(3;\bb C)$ via appropriate invariant subspaces
\begin{equation}
    \lie{sl}(3;\bb C) = \bigoplus_{\lambda \in {\lie{t}^2}^*}{\lie{sl}(3;\bb C)}_{\lambda},
\end{equation}
where
\begin{equation}
    {\lie{sl}(3;\bb C)}_{\lambda} := {\bigcap_{h\in\lie{t}^2}\ker\left(\mathrm{ad}(h)-\lambda(h)\,\mathrm{Id}_{\lie{sl}(3;\bb C)}\right)}.
\end{equation}
Hence, \emph{a roots' set} is nothing but $\Pi := \left\{\alpha \in {\lie{t}^2}^{*}\setminus \{0\} : {\lie{sl}(3;\bb C)}_{\alpha} \not\equiv 0  \right\}$. 

We can extract from $\Pi$ a subset $\Pi^{+}$ completely characterized by both:
\begin{enumerate}[(i)]
    \item for each root $\alpha\in \Pi$ only one of $\pm \alpha$ belongs to $\Pi^{+}$
    \item for each $\alpha, \beta \in \Pi^{+}$ necessarily $\alpha + \beta \in \Pi^{+}$ if $\alpha + \beta \in \Pi$.
\end{enumerate}
To the set $\Pi^{+}$, we name \emph{subset of positive roots}. We say that the subset $\Sigma \subset \Pi^{+}$ consists of the \emph{simple roots system} if it collects the positive roots, which can not be written as a combination of two elements in $\Pi^{+}$.

Since $\lie h$ is an Abelian Lie sub-algebra, we can pick a basis $\{ H_\alpha: \alpha \in \Sigma\}$ to $\lie h$ and complete it with $\{ X_\alpha \in {\mathfrak{sl}(3;\mathbb{C})}_\alpha: \alpha \in \Pi \}$, generating what is called a \textit{Weyl basis} of $\mathfrak{sl}(3;\mathbb{C})$: For any $\alpha, \beta \in \Pi$
\begin{enumerate}
    \item $\cal K(X_{\alpha},X_{-\alpha}) := \mathrm{tr}~({\mathrm{ad}(X_{\alpha})\circ \mathrm{ad}(X_{-\alpha})}) = 1$
    \item $[X_{\alpha},X_{\beta}] = m_{\alpha,\beta}X_{\alpha+\beta},~m_{\alpha,\beta}\in \mathbb{R}$,
\end{enumerate}
where $\cal K$ is the \emph{Cartan--Killing} form of $\lie g$, see \cite[p. 214]{San_Martin_2021}. Then $\mathfrak{sl}(3;\mathbb{C})$ decomposes into root subspaces:
\begin{equation*}
	\mathfrak{sl}(3;\mathbb{C})=\mathfrak{t}^2 + \sum_{\alpha \in \Pi} \mathbb{C} X_\alpha.
\end{equation*}
It is straightforward from \cite[Theorem 11.13, p. 224]{San_Martin_2021} that such a decomposition implies that the compact real form of $\mathfrak{su(3)}$ is given by
\begin{equation}\label{eq:compactrealform}
	\mathfrak{su(3)}=\sqrt{-1}\mathfrak{t}^2_{\mathbb{R}} + \sum_{\alpha \in \Pi^+} \mathrm{Span}_\mathbb{R} (\sqrt{-1}(X_\alpha + X_{-\alpha}), (X_\alpha - X_{-\alpha})).
\end{equation}

Given the above information, it can be directly checked that a basis for the Lie algebra of $\mathrm{SU}(3)$ is given by
$$
     \h\mathrm{diag}(2\mathrm{i},-\mathrm{i},-\mathrm{i})
     ,\ \h\mathrm{diag}(0,\mathrm{i},-\mathrm{i})
     , \ \h\mathrm{A}_{12}
     , \ \h\mathrm{S}_{12}
     , \ \h\mathrm{A}_{13}
     , \ \h\mathrm{S}_{13}
     , \ \h\mathrm{A}_{23}
     , \ \h\mathrm{S}_{23},
$$
% \begin{align*}
%      \mathrm{diag}(2i,-\mathrm{i},\mathrm{i})\\
%      \mathrm{A}_{12}\\
%      \mathrm{S}_{12}\\
%      \mathrm{A}_{13}\\
%      \mathrm{S}_{13}\\
%      \mathrm{diag}(0,\mathrm{i},-\mathrm{i})\\
%      \mathrm{A}_{23}\\
%      \mathrm{S}_{23},
% \end{align*}
where $\mathrm{S}_{kj}$ is a symmetric matrix $3\times 3$ with $\mathrm{i}$ in inputs $kj$ and $jk$ and $0$ in the others. On the other hand, $\mathrm{A}_{jk}$ is an antisymmetric matrix $3\times 3$ that has $1$ on input $kj$ and $-1$ on input $jk$, $0$ elsewhere. Moreover, $\mathrm{i} = \sqrt{-1}$. Hence, we can extract a basis for the tangent space $T_{o}\mathrm{SU}(3)/\mathrm{T}^2$ by disregarding the matrices $\mathrm{diag}(2i,-\mathrm{i}, -\mathrm{i})$ and $\mathrm{diag}(0, \mathrm{i},-\mathrm{i})$. Furthermore, the $3$ components of the isotropy representation are generated by \[\mathrm{span}_{\bb R}\left\{\h\mathrm{A}_{jk},\h\mathrm{S}_{jk}\right\},\]
for each pair $jk$.

 We digress a bit recalling that whenever a homogeneous space $M=G/K$ is \textit{reductive}, with reductive decomposition $\mathfrak{g}=\mathfrak{k}\oplus\mathfrak{m}$ (that is, $[\mathfrak{k},\mathfrak{m}]\subset \mathfrak{m}$), then $\lie m$ is $\Ad_G(K)$-invariant. Moreover, the map $\lie g \to T_b(G/K)$ that assigns to $X \in \lie g$ the induced tangent vector \[X \cdot b = d/dt (\exp(tX)b) |_{t=0}\] is surjective with kernel the isotropy subalgebra $\lie k$. Using that $g \in G$ acts in tangent vectors by its differential, we have that
\begin{equation}
\label{eq-induzido}
g ( X \cdot b ) = ( \Ad(g)X ) \cdot g b.
\end{equation}
Hence, the restriction $\lie m \to T_b(G/K)$ of the above map is a linear isomorphism that intertwines the isotropy representation of $K$ in $T_b(G/K)$ with the adjoint representation of $G$ restricted to $K$ in $\lie m$. This allows us to identify $T_b(G/K) = \lie m$ and the $K$-isotropy representation with the $\Ad_G(K)$-representation.

Being $G$ a compact connected simple Lie group such that the isotropy representation of $G/K$ decomposes $\mathfrak{m}$ as 
\begin{equation}\label{deco-iso}
\mathfrak{m}=\mathfrak{m}_1\oplus \ldots \oplus \mathfrak{m}_n
\end{equation}
where $\mathfrak{m}_1,\ldots,\mathfrak{m}_n$ are irreducible pairwise non-equivalent isotropy representations, as a consequence of Schur’s lemma, any invariant metrics are given by
\begin{align}
\label{eq-compon-metr}
g_b&=x_1B_1+\ldots + x_nB_n
\end{align}
where $x_i>0$ and $B_i$ is the restriction of the (negative of the) Cartan-Killing form of $\mathfrak{g}$ to $\mathfrak{m}_i$. According to Schur's lemma, we also have 
\begin{equation}
\label{eq-compon-ricci}
\Ricci (g_b)=y_1 B_1 + \ldots + y_nB_n
\end{equation}
where 
$y_i$ is a function of $x_1, \ldots, x_n$.

Turning back to our example, which is a reductive homogeneous space, considering the previous discussion, an $\mathrm{Ad}(\mathrm{T}^2)$-invariant inner product $\ga$ is determined by three parameters $(x,y,z)$ characterized by 
\begin{align*}
     g\left(\h\mathrm{A}_{12},\h\mathrm{A}_{12}\right) &= g\left(\h\mathrm{S}_{12},\h\mathrm{S}_{12}\right) = x,\\
     g\left(\h\mathrm{A}_{13},\h\mathrm{A}_{13}\right) &= g\left(\h\mathrm{S}_{13},\h\mathrm{S}_{13}\right) = y,\\
     g\left(\h\mathrm{A}_{23},\h\mathrm{A}_{23}\right) &= g\left(\h\mathrm{S}_{23},\h\mathrm{S}_{23}\right) = z.
\end{align*}
We then redefine new basis to $\mathfrak{m} := T_0(\mathrm{SU}(3)/\mathrm{T}^2)$ by
$$X_1 = \frac{1}{2\sqrt{x}}\mathrm{A}_{12}, \ X_2 = \frac{1}{2\sqrt{x}}\mathrm{S}_{12}, \ X_3 = \frac{1}{2\sqrt{y}}\mathrm{A}_{13}, \ X_4 = \frac{1}{2\sqrt{y}}\mathrm{S}_{13}, \ X_5 = \frac{1}{2\sqrt{z}}\mathrm{A}_{23}, \ X_6 = \frac{1}{2\sqrt{z}}\mathrm{S}_{23} .$$

Since the following formula holds for the sectional curvature of $\ga$ (see \cite[Theorem 7.30, p. 183]{besse})
\begin{align*}
K(X,Y) = -\frac{3}{4}\|[X,Y]_{\mathfrak{m}}\|^2 - \frac{1}{2}g([X,[X,Y]_{\lie g}]_{\mathfrak{m}},Y) - \frac{1}{2}g([Y,[Y,X]_{\lie g}])_{\mathfrak{m}},X)\\
+\|U(X,Y)\|^2 - \ga(U(X,X),U(Y,Y)),
\end{align*}
\begin{equation*}
2\ga(U(X,Y),Z) = \ga([Z,X]_{\lie m},Y) + \ga(X,[Z,Y]_{\lie m})
\end{equation*}
we can set up the following table, where $C_{ij}^k$ denotes a structure constant, that is, $C_{ij}^k = \ga([X_i, X_j], X_k)$, and $K_{ij}$ the sectional curvature. Moreover, that for $(i,j)\neq (1,2), (3,4), (5,6)$ it holds that
$$
    K(X_i,X_j) = K_{ij} = -\h C_{ij}^kC_{ik}^j -\h C_{kj}^iC_{ij}^k - \frac{3}{4}(C_{ij}^k)^2 + \sum_{l=1}^6\frac{1}{4}\left(C_{li}^j+C_{lj}^i\right)^2-\sum_{l=1}^6C_{li}^iC_{lj}^j.
$$
\begin{center}
\begin{table}[h!]
     \begin{tabular}{ | l | l | c | c | c |}
     \hline
     $\mathrm{i}$ & $j$ & $k$ & $C_{ij}^k$ & $K_{ij}$ \\ \hline\hline
     $1$ & $2$ & $\mathrm{diag}(\mathrm{i},-\mathrm{i},0)$ & $1/x$ & $1/x$\\ \hline
     $1$ & $3$ & $5$ & $-\frac{\sqrt{z}}{2\sqrt{xy}}$ & $-\frac{3}{16}\frac{z}{xy} + \frac{1}{8x} + \frac{1}{8y} + \frac{1}{16}\frac{(x-y)^2}{xyz}$\\ \hline
     $1$ & $4$ & $6$ & $-\frac{\sqrt{z}}{2\sqrt{xy}}$ & $-\frac{3}{16}\frac{z}{xy} + \frac{1}{8x} + \frac{1}{8y} + \frac{1}{16}\frac{(x-y)^2}{xyz}$ \\ \hline
     $1$ & $5$ & $3$ & $\frac{\sqrt{y}}{2\sqrt{xz}}$ & $-\frac{3}{16}\frac{y}{xz} + \frac{1}{8x} + \frac{1}{8z} + \frac{1}{16}\frac{(z-x)^2}{xyz}$ \\ \hline
     $1$ & $6$ & $4$ & $\frac{\sqrt{y}}{2\sqrt{xz}}$ & $-\frac{3}{16}\frac{y}{xz} + \frac{1}{8x} + \frac{1}{8z}+\frac{1}{16}\frac{(z-x)^2}{xyz}$ \\ \hline
     $2$ & $3$ & $6$ & $\frac{\sqrt{z}}{2\sqrt{xy}}$ & $-\frac{3}{16}\frac{z}{xy} + \frac{1}{8x} + \frac{1}{8y} + \frac{1}{16}\frac{(y-x)^2}{xyz}$ \\ \hline
     $2$ & $4$ & $5$ & $-\frac{\sqrt{z}}{2\sqrt{xy}}$ & $-\frac{3}{16}\frac{z}{xy} + \frac{1}{8x} + \frac{1}{8y} + \frac{1}{16}\frac{(y-x)^2}{xyz}$ \\ \hline
     $2$ & $5$ & $4$ & $\frac{\sqrt{y}}{2\sqrt{xz}}$ &   $-\frac{3}{16}\frac{y}{xz} + \frac{1}{8x} + \frac{1}{8z} + \frac{1}{8}\frac{(z-x)^2}{xyz}$ \\ \hline
     $2$ & $6$ & $3$ & $-\frac{\sqrt{y}}{2\sqrt{xz}}$ & $-\frac{3}{16}\frac{y}{xz} + \frac{1}{8x} + \frac{1}{8z} + \frac{1}{16}\frac{(z-x)^2}{xyz}$\\ \hline
     $3$ & $4$ & $\mathrm{diag}(\mathrm{i},0,-\mathrm{i})$ & $1/y$ & $1/y$ \\ \hline
     $3$ & $5$ & $1$ & $-\frac{\sqrt{x}}{2\sqrt{yz}}$ & $-\frac{3}{16}\frac{x}{yz} + \frac{1}{8y} +  \frac{1}{8z}+\frac{1}{16}\frac{(y-z)^2}{xyz}$ \\ \hline
     $3$ & $6$ & $2$ & $ \frac{\sqrt{x}}{2\sqrt{yz}}$ & $-\frac{3}{16}\frac{x}{yz} + \frac{1}{8y} +  \frac{1}{8z}+\frac{1}{16}\frac{(y-z)^2}{xyz}$ \\ \hline
     $4$ & $5$ & $2$ & $-\frac{\sqrt{x}}{2\sqrt{yz}}$ &$-\frac{3}{16}\frac{x}{yz} + \frac{1}{8y} +  \frac{1}{8z}+\frac{1}{16}\frac{(y-z)^2}{xyz}$ \\ \hline
     $4$ & $6$ & $1$ & $-\frac{\sqrt{x}}{2\sqrt{yz}}$ & $-\frac{3}{16}\frac{x}{yz} + \frac{1}{8y} +  \frac{1}{8z}+\frac{1}{16}\frac{(y-z)^2}{xyz}$ \\ \hline
     $5$ & $6$ & $ \mathrm{diag}(0,\mathrm{i},-\mathrm{i})$ & $1/z$ & $1/z$ \\ \hline
     \end{tabular}
     \caption{Structure Constants and Sectional curvature of the basis' elements \label{table:1}}
     \end{table}
     \end{center}

 We can compute every notion of positive curvature from Table \ref{table:1}. Particularly, the Ricci curvature formula is given by 
$\Ricci(X) = \sum_{\mathrm{i}=1}^nK(X,X_i)$, straightforward computations from Table \ref{table:1} leads to
\begin{align}\label{eq:ricci01}
\Ricci(X_1) = \Ricci(X_2) =&\frac{1}{2x} + \frac{1}{12}\left(\frac{x}{y z}-\frac{z}{x y}-\frac{y}{x z}\right),\\
\Ricci(X_3) = \Ricci(X_4) =&\frac{1}{2y} + \frac{1}{12}\left(-\frac{x}{y z}-\frac{z}{x y}+\frac{y}{x z}\right) ,\\  
\Ricci(X_5) = \Ricci(X_6)=&\frac{1}{2z}+ \frac{1}{12}\left(-\frac{x}{y z}+\frac{z}{x y}-\frac{y}{x z}\right).\label{eq:ricci02}
\end{align}
Next, we discuss different notions of intermediate positive Ricci curvature, furnishing a common ground for subsequent analyses and (hence) for the proof of Theorems \ref{ithm:main}, \ref{ithm:main3}.

\subsection{Conditions to \texorpdfstring{$\Ricd$}{Ricci-d} on the \texorpdfstring{$d\mathrm{th}$}{dth}-intermediate positive} Ricci curvatures of left-invariant metrics on \texorpdfstring{$\mathrm{SU}(3)/\mathrm{T}^2$}{SU(3)/T2} (interpolating between positive sectional and positive Ricci curvature)
\label{sec:def1}
We now take advantage of Table \ref{table:1} considering the symmetries appearing on the expressions for sectional curvature to get a simplified description of $d\mathrm{th}$-intermediate positive Ricci curvature (recall the Definition \ref{def:lm}). Take $d$-vectors out of the basis $\{X_1,\ldots,X_6\}\subset T_o\mathrm{SU}(3)/\mathrm{T}^2$ and pick any $1 \leq d\leq 5$. To describe properly the $\Ricci_d$-curvature on the direction of a given vector $X_i$ out of this basis, we must handle with some combinatorial quantities. To be more precise, observe that ensuring positivity of $\Ricci_d(X_i)$ in the sense of Definition \ref{def:lm}, is related to collecting every possible combination appearing as below
\begin{align*}
    \Ricci_d(X_i) = a_{12}\frac{1}{x} + b_{12}\left(-\frac{3}{16}\frac{z}{xy} + \frac{1}{8x} + \frac{1}{8y}+\frac{1}{16}\frac{(x-y)^2}{xyz}\right) + c_{12}\left(-\frac{3}{16}\frac{y}{xz} + \frac{1}{8x} + \frac{1}{8z} + \frac{1}{16}\frac{(z-x)^2}{xyz}\right),~i=1,2,\\
    \Ricci_d(X_i) = a_{34}\frac{1}{y} + b_{34}\left(-\frac{3}{16}\frac{z}{xy} + \frac{1}{8x} + \frac{1}{8y} + \frac{1}{16}\frac{(y-x)^2}{xyz}\right) + c_{34}\left(-\frac{3}{16}\frac{x}{yz} + \frac{1}{8y} +  \frac{1}{8z}+\frac{1}{16}\frac{(y-z)^2}{xyz}\right),~i = 3,4,\\
    \Ricci_d(X_i) = a_{56}\frac{1}{z} + b_{56}\left(-\frac{3}{16}\frac{y}{xz} + \frac{1}{8x} + \frac{1}{8z} + \frac{1}{16}\frac{(z-x)^2}{xyz}\right) + c_{56}\left(-\frac{3}{16}\frac{x}{yz} + \frac{1}{8y} +  \frac{1}{8z}+\frac{1}{16}\frac{(y-z)^2}{xyz}\right),~i = 5,6
\end{align*}
where $a_{jj+1} \in \{0,1\}$, $b_{jj+1}, c_{jj+1} \in \{0,1,2\}$ are such that $a_{jj+1} + b_{jj+1} + c_{jj+1} = d\leq 5$, with $(j,j+1)\in \{(1,2),(3,4),(5,6)\}$. That is, for instance, we have that $\Ricci_d(X_i)$ is positive for $i = 1, 2$ if, for every possible choice of $a_{12}\in \{0,1\},~b_{12}, c_{12} \in \{0,1,2\}$ we have $\Ricci_d(X_i) > 0$.

%%%%%%%%%%%%%%%%%%%%% versao antes da revisao
% We picture that in this setting, it suffices to obtain positive $\Ricci_d$ for every vector tangent to $\mathrm{SU}(3)/\mathrm{T}^2$ to look to the former expressions since $\Ricci_d\left(\sum_{i=1}^6x^iX_i\right) = \sum_{i=1}^6(x^i)^2\Ricci_d(X_i).$ %we have

%  Therefore, to ensure the existence of some $1\leq d \leq 5$ with positive $\Ricci_d$ curvature (in the sense of Definition \ref{def:lm}), it is necessary and sufficient to find such a $d$ constrained as: For every $(j,j+1)\in \{(1,2),(3,4),(5,6)\}$ and every $a_{jj+1} \in \{0,1\},~ b_{jj+1}, c_{jj+1} \in \{0,1,2\}$ with $a_{jj+1}+b_{jj+1}+c_{jj+1} = d$ it holds that $\Ricci_d(X_i) > 0$ for some $x, y, z$.
%%%%%%%%%%%%%

%%%%%%% novo texto sugerido
We observe that, in this setting, to ensure positive $\Ricci_d$ for every vector tangent to $\mathrm{SU}(3)/\mathrm{T}^2$, it suffices to consider the expressions above, since
\[
\Ricci_d\left(\sum_{i=1}^6 x^iX_i\right)
= \sum_{i=1}^6 (x^i)^2\Ricci_d(X_i).
\]
Therefore, to ensure the existence of some $1\leq d\leq 5$ with positive $\Ricci_d$ curvature (in the sense of Definition \ref{def:lm}), it is necessary and sufficient to find such a $d$ satisfying the following condition: for every $(j,j+1)\in{(1,2),(3,4),(5,6)}$ and every
\[
a_{jj+1}\in{0,1},\qquad b_{jj+1},c_{jj+1}\in{0,1,2},
\]
with
\[
a_{jj+1}+b_{jj+1}+c_{jj+1}=d,
\]
we have $\Ricci_d(X_i)>0$ for some $x,y,z$.

%%%%%%%

%As will be useful in Section \ref{sec:comparativo}, let us assume that $z = 1 - x - y$ and introduce the following functions
%\begin{align*}
 %   f(x,y) &:= -\frac{3(1-x-y)}{16xy} + \frac{1}{8x} + \frac{1}{8y}+\frac{1}{16}\frac{(x-y)^2}{xy(1-x-y)}\\
  %  g(x,y) &:= -\frac{3y}{16x(1-x-y)} + \frac{1}{8x} + \frac{1}{8(1-x-y)} + \frac{1}{16}\frac{(1-2x-y)^2}{xy(1-x-y)}\\
   % h(x,y) &:= -\frac{3x}{16y(1-x-y)} + \frac{1}{8y} +  \frac{1}{8(1-x-y)}+\frac{1}{16}\frac{(1-2y-x)^2}{xy(1-x-y)}
%\end{align*}
%We recover equation \eqref{eq:fullriccid} as
%\begin{multline}\label{eq:riccirewritten}
  %  \Ricci_d\left(\sum_{i=1}^6x^iX_i\right) = \sum_{i=1}^2(x^i)^2\left(\frac{a_{12}}{x} + %b_{12}f(x,y) + c_{12}g(x,y)\right) + \\\sum_{i=3}^4(x^i)^2\left(\frac{a_{34}}{y} + b_{34}f(x,y) %+ c_{34}h(x,y)\right) + \sum_{i=5}^6(x^i)^2\left(\frac{a_{56}}{1-x-y} + b_{56}g(x,y) + c_{56}%(x,y)\right).
%\end{multline}
Summarily, fixing the normalization of the parameters of the metric $x+y+z=1$, given by the normalized Ricci flow (see Section \ref{sec:proj-RF}), and fixed $d\in \{1,\ldots,5\}$, an invariant Riemannian metric $\ga = (x,y,1-x-y)$ in $\mathrm{SU}(3)/\mathrm{T}^2$ lies in $\Ricd$ if, and only if, for every $(j,j+1)\in \{(1,2),(3,4),(5,6)\}$, the scalar functions defined below, denote generically by $R^{jj+1}_{a,b,c}(x,y)$, are positive simultaneously for every $(a_{jj+1},b_{jj+1},c_{jj+1})\in \cal O_d := \{(a,b,c)\in \{0,1\}\times \{0,1,2\}^2 : a + b + c = d\}$
\begin{align}
   \label{eq:explicit1} R^{12}_{a,b,c}(x,y) &:= a_{12}\frac{1}{x} + b_{12}\left(-\frac{3}{16}\frac{z}{xy} + \frac{1}{8x} + \frac{1}{8y}+\frac{1}{16}\frac{(x-y)^2}{xyz}\right) + c_{12}\left(-\frac{3}{16}\frac{y}{xz} + \frac{1}{8x} + \frac{1}{8z} + \frac{1}{16}\frac{(z-x)^2}{xyz}\right)\\
   \label{eq:explicit2} R^{34}_{a,b,c}(x,y) &:= a_{34}\frac{1}{y} + b_{34}\left(-\frac{3}{16}\frac{z}{xy} + \frac{1}{8x} + \frac{1}{8y} + \frac{1}{16}\frac{(y-x)^2}{xyz}\right) + c_{34}\left(-\frac{3}{16}\frac{x}{yz} + \frac{1}{8y} +  \frac{1}{8z}+\frac{1}{16}\frac{(y-z)^2}{xyz}\right)\\
   \label{eq:explicit3} R^{56}_{a,b,c}(x,y) &:= a_{56}\frac{1}{z} + b_{56}\left(-\frac{3}{16}\frac{y}{xz} + \frac{1}{8x} + \frac{1}{8z} + \frac{1}{16}\frac{(z-x)^2}{xyz}\right) + c_{56}\left(-\frac{3}{16}\frac{x}{yz} + \frac{1}{8y} +  \frac{1}{8z}+\frac{1}{16}\frac{(y-z)^2}{xyz}\right)
\end{align}

\subsection{Conditions to \texorpdfstring{$\ricd$}{ricd}: On the intermediate positive Ricci curvatures of left-invariant metrics on \texorpdfstring{$\mathrm{SU}(3)/\mathrm{T}^2$}{SU(3)/T2} (interpotaling between positive Ricci and positive scalar curvature)}
\label{sec:def2}

Let us denote $\Ricci_{\ga}(X_1) = \Ricci_{\ga}(X_2) := r_x,~\Ricci_{\ga}(X_3) = \Ricci_{\ga}(X_4) := r_y$ and $\Ricci_{\ga}(X_5) = \Ricci_{\ga}(X_6) := r_z$. One recovers
\begin{eqnarray*}
r_x&=&\frac{1}{2 x} +\frac{1}{12} \left(\frac{x}{y z}-\frac{z}{x y}-\frac{y}{x z}\right),\\ \\
r_y&=&\frac{1}{2 y} + \frac{1}{12} \left(-\frac{x}{y z}-\frac{z}{x y}+\frac{y}{x z}\right), \\ \\ 
r_z&=&\frac{1}{2 z}+ \frac{1}{12} \left(-\frac{x}{y z}+\frac{z}{x y}-\frac{y}{x z}\right)
\end{eqnarray*}

Let $a, b, c \in \bb N$ non-negative integers. Following Definition \ref{def:ours}, in Section \ref{sec:comparativo}, we shall deal with positive intermediate Ricci curvature ranging from positive Ricci to positive scalar curvature. Aiming such a goal, let $d\in \{1,\ldots,6\}$ and consider the set  $\cal N_d:= \{(a, b, c) \in \{0,1,2\}^3: a+b+c = d\}$. Define the scalar function $F_{a,b,c}(x,y,z):= ar_x + br_y + cr_z$.  If for every $(a,b,c) \in \cal N_d$ one has $F_{a,b,c}(x,y,z) > 0$, we say that $\ga$ has \emph{$d$-positive intermediate curvature} and $(x,y,z) \in \cal \ricd$.

\section{The projected homogeneous Ricci flow: Quick overview and the equations in \texorpdfstring{$\mathrm{SU}(3)/\mathrm{T}^2$}{SU(3)/T2}}\label{sec:proj-RF}
We recall that a family of Riemannian metrics $g(t)$ in $M$ is called a Ricci flow if it satisfies
\begin{equation}
\label{ricci-flow}
\frac{\partial \textsl{g}}{ \partial t}=-2\Ricci(\textsl{g}).
\end{equation}
%It can be checked that for any $\lambda >0$ we have $\Ricci(\lambda \ga) = \Ricci(\ga)$. Consequently, the \emph{Ricci operator} $r(\ga)$, given by
%\begin{equation}
%\label{eq-ricciop}
%\Ricci(\ga)(X, Y) = \ga(r(\ga)X, Y)
%\end{equation}
%is homogeneous of degree $-1$: $r(\lambda \ga) = \lambda^{-1} r(\ga)$, and so is the scalar curvature $S(g) = \rm{tr}(r(\ga))$. Moreover,
% For any compact connected and $n$-dimensional manifold $M$ one can consider (see \cite{bwz}):
% \begin{equation}
% \frac{d\ga_b}{dt}=-2\left(\Ricci(\ga_b) - \frac{\mathrm{S}(\ga_b) }{n}{\ga}_b \right)
% \end{equation}
% which preserves the metrics with unit volume and is the gradient flow of $\ga_b\mapsto \mathrm{S}(\ga_b)$ when restricted to such space.  
% In particular, the normalized Ricci flow
% \begin{equation}
% \label{normaliz-ricci-flow}
% \frac{\partial \ga}{ \partial t}=-2\left( \Ricci(\ga) - \frac{T(\ga)}{n}\ga \right)
% \end{equation}
% that preserves metrics of unit volume $V(\ga) = \int_M dV_{\ga}$ necessarily decreases scalar curvature; where $dV_{\ga}$ is the Riemannian volume form and $T(\ga) = \int_M\mathrm{S}(\ga) dV_{\ga}$ is the total scalar curvature functional.
For any compact homogeneous space $M=G/K$ with connected isotropy subgroup $K$, a $G$-invariant metric $\textsl g$ on $M$ is determined by its value $\textsl{g}_b$ at the origin $b=K$, which is a $\Ad_G(K)$-invariant inner product. Just like $\ga$, the  Ricci tensor $\Ricci (\ga)$ and the scalar curvature $\mathrm{S}(\ga)$ are also $G$-invariant and completely determined by their values at $b$, $\Ricci(\ga)_b = \Ricci(\ga_b)$, $\mathrm{S}(\ga)_b=\mathrm{S}(\ga_b)$. Taking this into account, the Ricci flow equation (\ref{ricci-flow}) becomes the autonomous ordinary differential equation known as the (non-normalized) \textit{homogeneous Ricci flow}:
\begin{equation}
\label{inv-ricci-flow}
\frac{dg_b}{dt}=-2\Ricci(g_b).
\end{equation}
% The equilibria of (\ref{normaliz-ricci-flow}) are precisely the metrics satisfying $\Ricci(\ga) = \lambda \ga$, $\lambda \in \mathbb{R}$, the so called \emph{Einstein metrics}.  On the other hand, the unit volume Einstein metrics are precisely the critical points of the functional $\mathrm{S}(\ga)$  on the space of unit volume metrics (see \cite{wz}). 
Recalling equation \eqref{eq-compon-metr} and \eqref{eq-compon-ricci} one derives that the Ricci flow (\ref{inv-ricci-flow}) becomes the autonomous system of ordinary differential equations
\begin{equation}
\label{eq-ricci-flow-coords}
\frac{dx_k}{dt}= -2 y_k, \qquad \ k=1,\ldots , n.
\end{equation}

It is always very convenient to rewrite the Ricci flow equation in terms of the {Ricci operator} $r(g)_b$, which is possible since $r(g)_b$ is invariant under the isotropy representation and hence
$r(g)_b|_{\mathfrak{m}_k}$ is a multiple $r_k$ of the identity. One obtains
$$
y_k = x_k r_k
$$
and equation \eqref{eq-ricci-flow-coords} becomes
\begin{equation}
\label{eq-ricci-flow-final}
\frac{dx_k}{dt}= -2 x_kr_k.
\end{equation}

Recalling that the isotropy representation of $\mathrm{SU}(3)/\mathrm{T}^2$ decomposes into three irreducible and non-equivalent components:
$$
\mathfrak{m}=\mathfrak{m}_1 \oplus \mathfrak{m}_2 \oplus \mathfrak{m}_3.$$

 The Ricci tensor of an invariant metric $\ga = (x,y,z)$ is also invariant, and its components are given by (recall equations \eqref{eq:ricci01}-\eqref{eq:ricci02}):

\begin{eqnarray*}
r_x&=&\frac{1}{2 x} +\frac{1}{12} \left(\frac{x}{y z}-\frac{z}{x y}-\frac{y}{x z}\right),\\ \\
r_y&=&\frac{1}{2 y} + \frac{1}{12} \left(-\frac{x}{y z}-\frac{z}{x y}+\frac{y}{x z}\right), \\ \\ 
r_z&=&\frac{1}{2 z}+ \frac{1}{12} \left(-\frac{x}{y z}+\frac{z}{x y}-\frac{y}{x z}\right)
\end{eqnarray*}
and the corresponding (unnormalized) Ricci flow equation is given by 
\begin{equation}\label{RF-eq}
 x^\prime = -2x r_x, \quad y^\prime = -2y r_y, \quad z^\prime = -2z r_z.   
\end{equation}

The projected Ricci flow is obtained by a suitable reparametrization of the time, obtaining an induced system of ODEs with phase-portrait on the set
$$
\{(x,y,z)\in \mathbb{R}^3: x+y+z=1  \}\cap {\bb R}^3_+,
$$
where $\mathbb{R}^3_+=\{ (x,y,z)\in \mathbb{R}^3: x>0,y>0,z>0  \}$, following of the projection on the $xy-$plane. The resulting system of ODEs is dynamically equivalent to the system \eqref{RF-eq} (Theorem B in \cite{proj-ricci-flow}).

Applying the analysis developed in \cite[Section 5]{proj-ricci-flow}, we arrive at the equations of the projected Ricci flow equation (see equation (20) in Section 5 of \cite{proj-ricci-flow}, setting $m=n=p=1$ on such equation):
\begin{equation}\label{proj-ricci-flow}
    \left\{\begin{array}{l}
    x^{\,\prime} = {u(x,y)},\\
    y^{\,\prime} = {v(x,y)},
\end{array}\right.
\quad (x,y)\in T=\{(x,y)\in\R^2:\, x\geq 0,\,y\geq 0,\,x+y\leq 1\},
\end{equation}
where $$u(x,y)=2 x \left(x^2 (2-12 y)-3 x \left(4 y^2-6 y+1\right)+6 y^2-6 y+1\right),$$ and $$v(x,y)=-2 y (2 y-1) \left(6 x^2+6 x (y-1)-y+1\right).$$

\section{The global dynamics of invariant metrics on \texorpdfstring{$\mathrm{SU}(3)/\mathrm{T}^2$}{SU(3)/T^2} under the Ricci flow}
\label{sec:comparativo}

Here we accomplish the proof of Theorems \ref{ithm:main}--\ref{ithm:main3}. With this goal, let us begin recalling from Section \ref{sec:def1} that, fixed $d\in \{1,\ldots,5\}$, then $\ga \in \Ricd$ 
if, and only if, the scalar functions $R^{jj+1}_{(a,b,c)}(x,y)$ obtained from $\ga$ are simultaneously positive for every $(a,b,c)\in \cal O_d = \{(a,b,c)\in \{0,1\}\times \{0,1,2\}^2 : a + b + c = d\}$ and every $(j,j+1)\in \{(1,2),(3,4),(5,6)\}$.

In contrast to it, understanding the existence of positively curved metrics for the curvature notions interpolating between Ricci and scalar curvature, that is, fixed $d\in \{1,\ldots,6\}$, finding $\ga \in \ricd$, is reduced to, following Section \ref{sec:def2}, obtain $(x,y,z)$ for which
$F_{a,b,c}(x,y,z) > 0$ for every $a,b,c \in \cal N_d$, where
$\cal N_d:= \{(a, b, c) \in \{0,1,2\}^3: a+b+c = d\}$,~ $F_{a,b,c}(x,y,z) := ar_x + br_y + cr_z$ and
\begin{eqnarray*}
r_x&=&\frac{1}{2 x} +\frac{1}{12} \left(\frac{x}{y z}-\frac{z}{x y}-\frac{y}{x z}\right),\\ \\
r_y&=&\frac{1}{2 y} + \frac{1}{12} \left(-\frac{x}{y z}-\frac{z}{x y}+\frac{y}{x z}\right), \\ \\ 
r_z&=&\frac{1}{2 z}+ \frac{1}{12} \left(-\frac{x}{y z}+\frac{z}{x y}-\frac{y}{x z}\right).
\end{eqnarray*}
% In what follows, we describe the global dynamics of the vector field $F(x,y)=\big(u(x,y),v(x,y)\big)$ associated with differential system \eqref{proj-ricci-flow}. 
In what follows, we describe the global dynamics of the vector field
\[
F(x,y)=\big(u(x,y),v(x,y)\big).
\]
This vector field is associated with the differential system
\eqref{proj-ricci-flow}.

The triangular domain  $\mathcal{T}$  can  be divided into four invariant triangles (see Figure \ref{phaseportrait}), namely:
\[
\begin{aligned}
 \mathcal{T}_1=&\{(x,y)\in T:\, x+y\leq 1/2\},\\
\mathcal{T}_2=&\{(x,y)\in T:\, x\geq 1/2\},\\
 \mathcal{T}_3=&\{(x,y)\in T:\, y\geq 1/2\},
\end{aligned}
\]
and the central one 
$$
\mathcal{T}_c=\{(x,y)\in T:\,x\leq1/2,\, y\leq 1/2,\,x+y\geq 1/2\}.
$$
Denote the vertices of $\mathcal{T}$  by  $O=(0,0)$, $P=(1,0)$, and $Q=(0,1).$ Also, denote the vertices of $\mathcal{T}_c$ by $L=(1/2,0),$ $M=(1/2,1/2),$ and $N=(0,1/2).$

The vertices $O, P, Q$ correspond to unstable star node equilibria. Indeed, for each $X\in\{O, P, Q\},$ the Jacobian matrix $dF(X)$ coincides with the identity matrix multiplied by $2$. The vertices $L, M, N$ correspond to stable star node equilibria. Indeed, for each $X\in\{L, M, N\},$ the Jacobian matrix $dF(X)$ coincides with the identity matrix multiplied by $-1$. Therefore, for each $X\in\{L, M, N, O, P, Q\},$ there exists a neighborhood $V\subset\R^2$ of $X$ such that vector field $F|_V$ is $C^1$ conjugated to the linear vector field $dF(X)\cdot (x,y).$ In particular, the closure of the orbits approaching to each one of these equilibria are transversal to each other at the equilibria. In addition, it is straightforward to see that heteroclinic orbits connect the unstable node $O$ to the stable nodes $L$ and $N$, the unstable node $P$ to the stable nodes $L$ and $M$, and the unstable node $Q$ to the stable nodes $M$ and $N$ (see Figure \ref{phaseportrait}).

Besides the vertices $L, M, N, O, P, Q$, the vector field $F$ has four other equilibria. Three of them belonging to the sides of  $\mathcal{T}_c$ and the last one in the interior of $\mathcal{T}_c,$ namely $R=(1/4,1/4)$, $S=(1/2,1/4)$,  $T=(1/4,1/2),$ and $U=(1/3,1/3).$ We remark that the points $R,S,T$ represents the K\"ahler-Einstein metrics; and the point $U$ represents the normal-Einstein metric. Notice that this metric has {nonnegative sectional curvature, but not positive sectional curvature} \cite{wallach, Abiev2016}.

The point $U$ corresponds to an unstable star node equilibrium. Indeed, the Jacobian matrix $dF(X)$ coincides with the identity matrix multiplied by $2/9$. Thus, the same comment above about the local $C^1$ conjugacy holds for $U$.

The points $R, S,$ and $T$ correspond to saddle equilibria. Indeed, for each $X\in\{R,S,T\},$ the Jacobian matrix $dF(X)$ has the eigenvalues $1/2$ and $-1/4$. In addition: the segments $\ov{OU}$, $\ov{PU}$, and $\ov{QU}$ correspond to the stable manifolds of $R,$ $S,$ and $T$, respectively; and the segments $\ov{NL}$, $\ov{LM}$, and $\ov{MN}$ correspond to the unstable manifolds of $R,$ $S,$ and $T$, respectively. In particular, the segments 
$\ov{RO},$ $\ov{RU},$ $\ov{SP},$ $\ov{SU},$ $\ov{TQ},$ and $\ov{TU}$ correspond to heteroclinic orbits connecting the unstable nodes with the saddles through the stable manifold; and the segments $\ov{RL},$ $\ov{RN},$ $\ov{SM},$ $\ov{SL}$, $\ov{TM},$ and $\ov{TN}$ correspond to heteroclinic orbits connecting the stable nodes with the saddles through the unstable manifold (see Figure \ref{phaseportrait}).

Let $\alpha(X)$ and $\omega(X)$ denote, respectively, the $\alpha$ and $\omega$ limit sets of $X\in\mathcal{T}.$ Using Poincar\'{e}--Bendixson Theorem arguments, one can easily see that $\alpha(X)=\{O\}$ for any $X\in\mathcal{T}_1^{\circ}$; $\alpha(X)=\{P\}$ for any $X\in\mathcal{T}_2^{\circ}$; $\alpha(X)=\{Q\}$ for any $X\in\mathcal{T}_3^{\circ}$; and $\alpha(X)=\{U\}$ for any $X\in\mathcal{T}_c^{\circ}$.  Also, one can easily see that $\omega(X)=\{N\}$ for any $X$ in the interior of the triangle $OUQ$; $\omega(X)=\{L\}$ for any $X$ in the interior of the triangle $OPU$; and $\omega(X)=\{M\}$ for any $X$ in the interior of the triangle $PQU$.

With the considerations above, we have completely described the asymptotic behavior of the trajectories with initial conditions lying on $\mathcal{T}$ (see Figure \ref{phaseportrait}).
\begin{figure}[H]
\centering
\begin{overpic}[height=6cm]{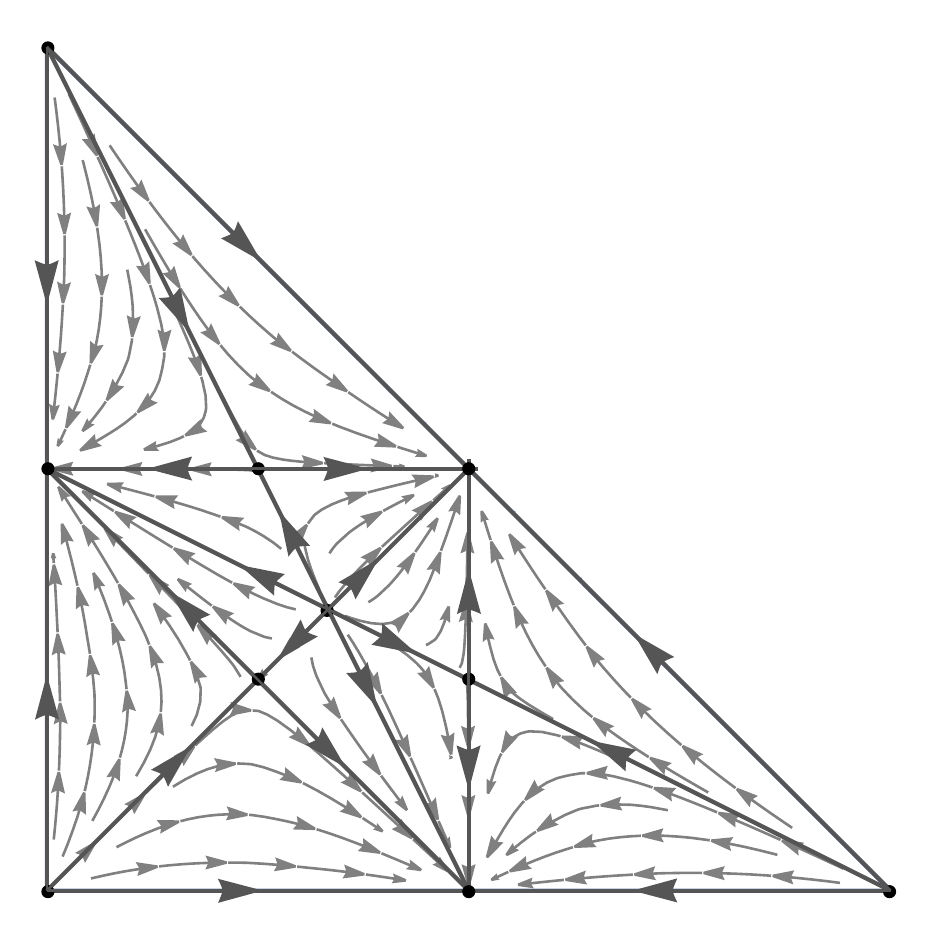}
%\begin{overpic}[height=6cm,grid,tics=5,]{phaseportrait.pdf}
\put(0,0){$O$}
\put(95,0){$P$}
\put(3,97){$Q$}
\put(47,0){$L$}
\put(51,52){$M$}
\put(-1,50){$N$}
\put(25,20){$R$}
\put(51,28){$S$}
\put(27,52){$T$}
\put(34,29){$U$}
\end{overpic}
\caption{Phase portrait of the vector field $F$. Black circles correspond to the equilibria, whereas the continuous segments connecting them correspond to the heteroclinic orbits. The $\alpha$ and $\omega$ limit set of $F$ for any initial condition are completely described above.}\label{phaseportrait}
\end{figure}

\subsection{Proof of Theorem \ref{ithm:main}}

Let $d\in \{1,2,3,4,5\}$ be fixed. Consider the family of scalar functions \[R^{12}_{a,b,c}(x,y),~R^{34}_{a,b,c}(x,y),~R^{56}_{a,b,c}\]
explicitly defined in equations \eqref{eq:explicit1}-\eqref{eq:explicit3}.

Accordingly, for each $d \in \{1,2,3,4,5\}$, the region $\Ricd$ can be visualized by drawing, over the triangle $\mathcal{T}$, the curves $R^{jj+1}_{a,b,c}(x,y)=0,$
for $(a,b,c) \in \mathcal{O}_d$ and $(j,j+1) \in \{(1,2),(3,4),(5,6)\}$. These curves correspond to the red curves in Figure~\ref{fig:sec-ric}. Since
\[
\Ricd=\bigcup_{(j,j+1)\in \{(1,2),(3,4),(5,6)\}}\{(x,y)\in\mathcal{T}:\, R^{jj+1}_{a,b,c}(x,y)>0\,\,\forall\,\,(a,b,c)\in \mathcal{O}_d\},
\]
one can see that the region $\Ricd$ corresponds to the interior of the union of the colored and checkered regions in Figure~\ref{fig:sec-ric}.

For each $d\in\{1,2,3\},$ $\Ricd$ is contained in the interior of the central triangle, $\mathcal T_c^{\circ}$. 

In addition, the boundary components of $\Ricd$ are tangent to one another at
the points $L$, $M$, and $N$. To see this, consider the case $d=1$.
The boundary of $\Ricd$ is determined by the zero sets of
$-4x^2+(1-2y)^2+4xy=0$,
$(1-2x)^2+4xy-4y^2=0,$
and
$-3+8y-4\bigl((-2+x)x+3xy+y^2\bigr)=0.$
Consider, for example, the first curve, which connects $L$ to $M$. By
computing its gradient at $L$ and $M$, and taking the orthogonal vector, one verifies that the tangent vectors of the curve at $L$ and $M$ are
eigenvectors of the Jacobian matrix of the vector field $F$ at each of these
points. Consequently, this curve is tangent to the invariant sets $\ov{UL}$ and
$\ov{UM}$ at $L$ and $M$, respectively. The same argument applies to the remaining boundary components, implying that they are tangent to one another at these points. It also applies to the other values of $d$, namely $2$ and $3$.

% In addition, the boundaries of $\Ricd$ are tangent to each other at the points $L, M,$ and $N.$ To see that, for instance, for $d=1$ we notice that $\Ricd$ is given by the union of the zero set of the equations
% $-4 x^2 + (1 - 2 y)^2 + 4 x y=0$,
% $(1 - 2 x)^2 + 4 x y - 4 y^2=0$,
% and $-3 + 8 y - 4 ((-2 + x) x + 3 x y + y^2)$. Taking, for instance, the first one, $-4 x^2 + (1 - 2 y)^2 + 4 x y=0$, that connects $L$ to $M$, by computing the gradient at the points $L$ and $M$, one can see that the tangent vector is an eigenvalue of de Jacobian of the vector field $F$ at these points. This means that the curve $-4 x^2 + (1 - 2 y)^2 + 4 x y=0$ is tangent to the invariant set $QL$ and $OM$ at these points. The same reasoning can be done for the other boundaries e also for the other values of $d.$

Moreover, at these points, such boundaries are tangent to the closure of the heteroclinic orbits $\ov{UL}$, $\ov{UM},$ and $\ov{UN}$ of 
 $F$.

 Since, for each $\ga=(x_0,y_0)\in \Ricd\setminus(\ov{UL}\cup\ov{UM}\cup\ov{UN})$, $\omega(x_0,y_0)\in\{L,M,N\},$ then there must exists exists $t^*\in  \R$ such that $\ga(t^*) \not\in \Ricd$. Otherwise, the closure of the orbit of $F$ through $(x_0,y_0)$ would be tangent to either $\ov{UL}$, $\ov{UM},$ or $\ov{UN}$ at $L,$ $M,$ or $N$, which is  impossible because $L,$ $M,$ and $N$ are star nodes equilibria and, as mentioned previously,  the closure of the orbits approaching to each one of these equilibria are transversal to each other at the equilibria. It is worth mentioning that for $g=(x_0,y_0)\in (\ov{UL}\cup\ov{UM}\cup\ov{UN})\setminus\{L,M,N\}$, $g(t)\in \Ricd$ for every $t\in\R.$

Now, for $d=4,$ $\Ricd$ coincides with the interior of the central triangle,  $\mathcal T_c^{\circ}$ (see Figure \ref{fig:sec-ric}), which is invariant by the flow of $F.$ Therefore, for each $g=(x_0,y_0)\in \Ricd$, one has that $g(t)\in \Ricd$ for every $t\in\R.$ 

Finally, for $d=5,$ $\Ricd$ contains the interior of the central triangle,  $\mathcal T_c^{\circ}$  (see Figure \ref{fig:sec-ric}). Moreover, $\Ricd\setminus \mathcal{T}_c$ is nonempty. Thus,  by taking $(x_0,y_0)\in \Ricd\setminus \mathcal{T}_c,$ since $\alpha(x_0,y_0)\in \{O,P,Q\},$ we conclude that there exists $t^*\in  \R$ such that $\ga(t) \not\in \ricd$ for every $t$ in a neighborhood of $t^*$.

\begin{figure}[!ht]

\begin{center}
   
\begin{tabular}{ccc}
    \includegraphics[height=52mm]{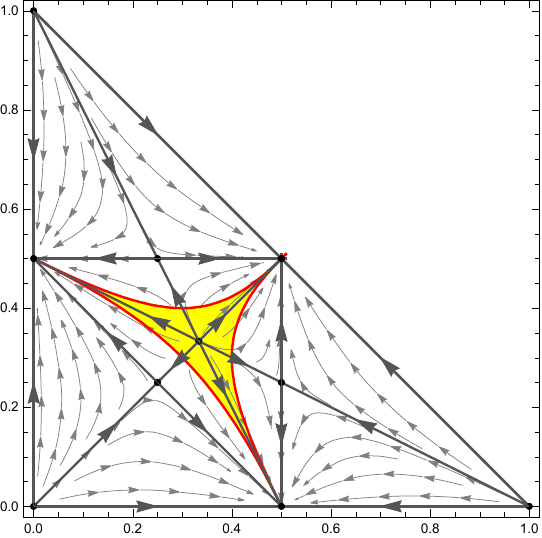} &\includegraphics[height=52mm]{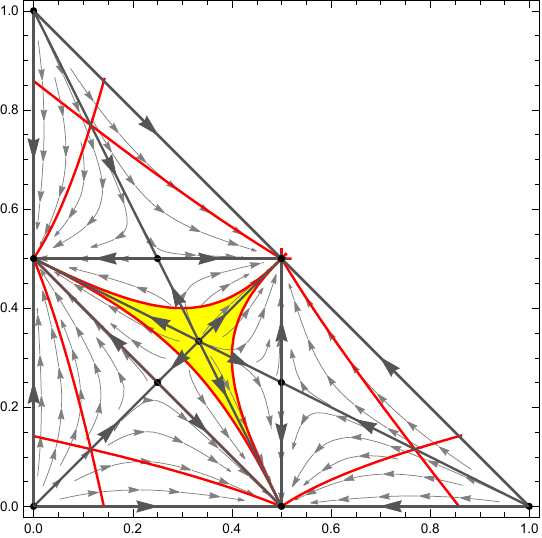}&\includegraphics[height=52mm]{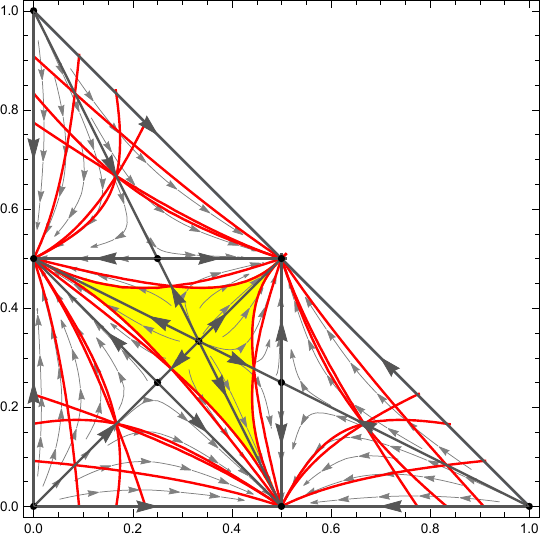}\\
    $d=1$& $d=2$&$d=3$\\
    \\
     \includegraphics[height=52mm]{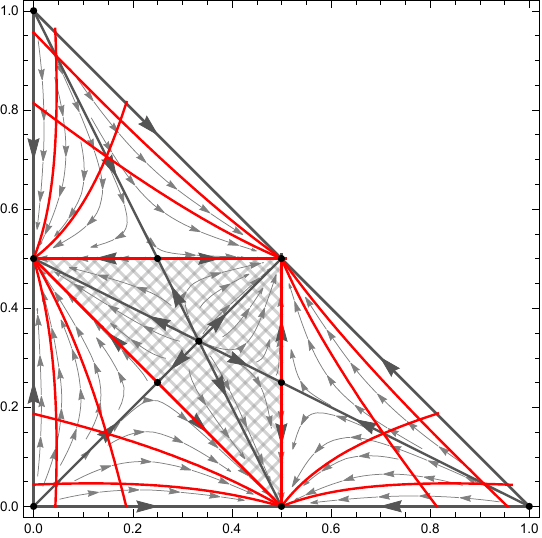} &\includegraphics[height=52mm]{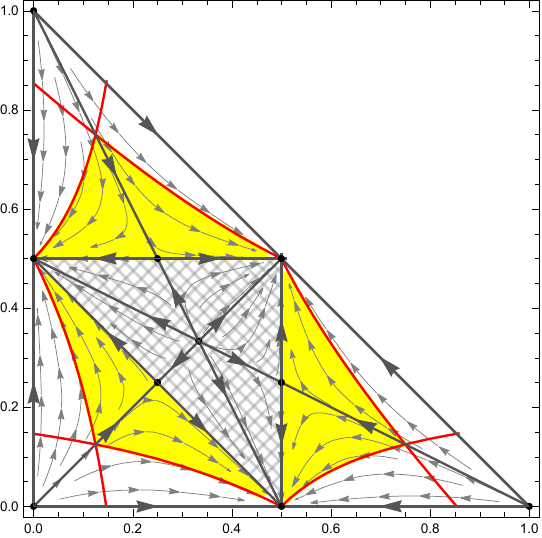}\\
     $d=4$&$d=5$
     % \label{fig:sec}
\end{tabular}

 \caption{Phase portrait of the Projected Ricci flow is depicted jointly with the regions $\Ricd$ (see Definition \ref{defi:sec-ric}), which correspond to the interior of the union of the yellow and checkered regions. The yellow regions represent the metrics $g\in \Ricd$ that lose this property under the Ricci flow. Meanwhile, the checkered region ($d=4,5$) represents the metrics $g\in \Ricd$ for which the projected Ricci flow $g(t)$ with $g(0)=g$ satisfies $g(t)\in \Ricd$ for all $t\in \mathbb{R}$. The red curves correspond to the curves $R^{jj+1}_{a,b,c}(x,y) = 0$, for $(a,b,c)\in \cal O_d$ and $(j,j+1)\in \{(1,2),(3,4),(5,6)\}$.}
    \label{fig:sec-ric}
\end{center}
\end{figure}

\subsection{Proof of Theorem \ref{ithm:main3}}

Let $d\in\{1,2,\ldots,6\}$ be fixed. For each $(a,b,c)\in\mathcal{N}_d$, we define the function $h_{a,b,c}(x,y):=F_{a,b,c}(x,y,1-x-y)$. 
Accordingly, for each $d \in \{1,2,3,4,5\}$, the region $\ricd$ can be visualized by drawing, over the triangle $\mathcal{T}$, the curves $h_{a,b,c}=0$, for each $(a,b,c)\in\mathcal{N}_d$. These curves correspond to the red curves in Figure~\ref{fig:ric-scal}. Since
\[
\ricd=\{(x,y)\in\mathcal{T}:\, h_{a,b,c}(x,y)>0\,\,\forall\,\,(a,b,c)\in \mathcal{N}_d\}.
\]
one can see that the region $\ricd$ corresponds to the interior of the union of the colored and checkered regions in Figure~\ref{fig:ric-scal}. 

One can see that the interior of the central triangle, $\mathcal{T}_c^{\circ}$, is contained in $\ricd$. Since $\mathcal{T}_c^{\circ}$  invariant by the flow of $F$, then statement (b) follows by taking $\mathcal{R}=\mathcal{T}^{\circ}$. Finally, by taking $(x_0,y_0)\in \ricd\setminus \mathcal{T},$ since $\alpha(x_0,y_0)\in \{O,P,Q\},$ we conclude that there exists $t^*\in  \R$ such that $\ga(t) \not\in \ricd$ for every $t$ in a neighborhood of $t^*$.

\begin{figure}[h!]
   
\begin{center} 
\begin{tabular}{ccc}
    \includegraphics[height=52mm]{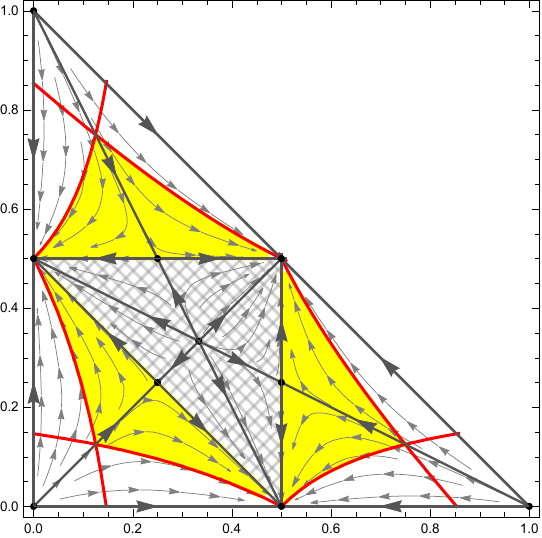} &\includegraphics[height=52mm]{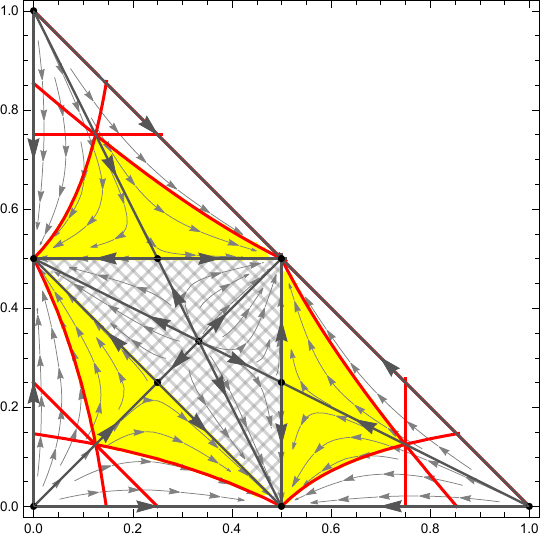}&\includegraphics[height=52mm]{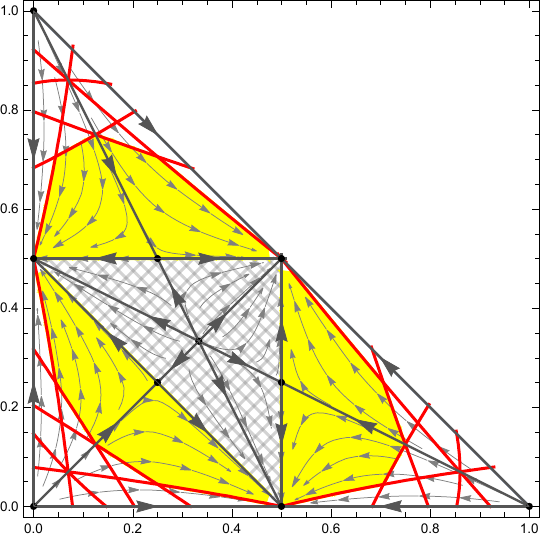}\\
    $d=1$& $d=2$&$d=3$\\
    \\
     \includegraphics[height=52mm]{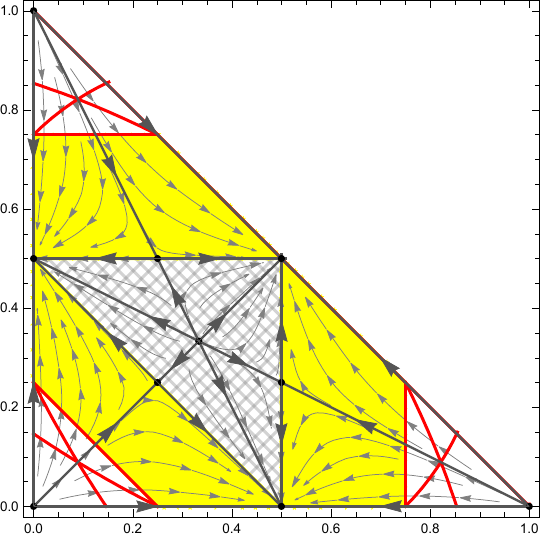} &\includegraphics[height=52mm]{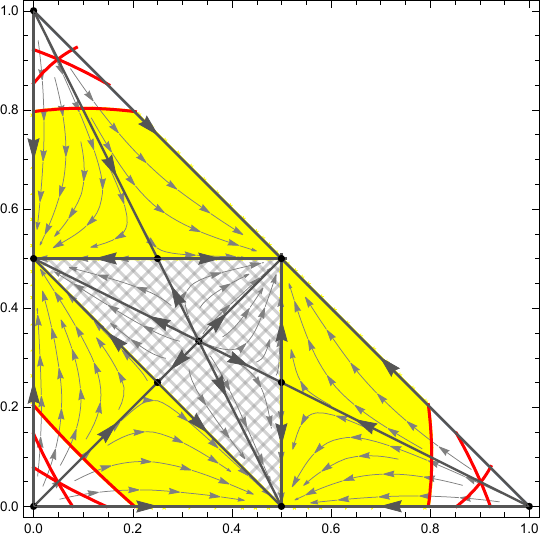}&\includegraphics[height=52mm]{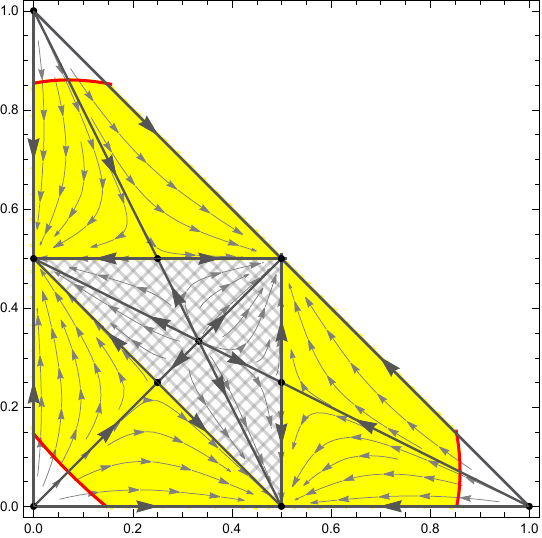}\\
     $d=4$&$d=5$&$d=6$
\end{tabular}

  \caption{Phase portrait of the Projected Ricci flow is depicted jointly with the regions $\ricd$ (see Definition \ref{defi:ric-scal}), which correspond to the interior of the union of the yellow and checkered regions. The yellow regions represent the metrics $g\in \Ricd$ that lose this property under the Ricci flow. Meanwhile, the checkered inner triangle represents the metrics $g\in \ricd$ for which the projected Ricci flow $g(t)$ with $g(0)=g$ satisfies $g(t)\in \ricd$ for all $t\in \mathbb{R}$. The red curves correspond to the curves $h_{a,b,c}=0$, for $(a,b,c)\in\mathcal{N}_d$.}
    \label{fig:ric-scal}

\end{center}
\end{figure}

\section{Associated bundles and intermediate positive curvatures}\label{sec:bundles}

	Let $G$ be a compact Lie group, let $\pi_P\colon P\to B$ be a principal $G$-bundle over a compact base $B$, and let $G$ act smoothly on a compact connected manifold $F$. Out of these data one forms the \emph{associated bundle}
	\[
	M := P\times_G F := (P\times F)/G,\qquad \pi\colon M\to B,\qquad \pi([p,f]) := \pi_P(p),
	\]
	which is a fiber bundle over $B$ with fiber $F$ \cite[Section I.5]{KN}. Its fibers are embedded copies of $F$. When the action on $F$ is effective, this construction produces the fiber bundles with structure group $G$: the transition functions of a suitable atlas of local trivializations of $\pi$ take values in $G$, and $P$ is recovered from this cocycle by \cite[Proposition 5.2]{KN}. We do \emph{not} assume effectivity, since none of the constructions below requires it and since, as explained in Remark \ref{rem:effectivity}, the hypotheses of Theorem \ref{thm:associated} are incompatible with it.
    
    It is straightforward to check that $\overline\pi\colon P\times F\to M$ is a principal $G$-bundle, whose principal action is the diagonal one \cite[Proposition 2.7.1]{gw}
	\begin{equation}\label{eq:action}
	r(p,f) := (rp, rf).
	\end{equation}
	
	For each pair $\ga$ and $\ga_F$ of $G$-invariant metrics on $P$ and $F$, respectively, there is a unique metric $\aga$ on $M$ for which $\overline\pi\colon (P\times F,\ga+\ga_F)\to (M,\aga)$ is a Riemannian submersion. Fixing a point $(p,f) \in P\times F$, any vector $\overline{X}\in T_{(p,f)}(P\times F)$ can be written as $\overline{X} = (X+V^{\vee},X_F+W^*)$, where $X$ is orthogonal to the $G$-orbit through $p$, $X_F$ is orthogonal to the $G$-orbit through $f$ and, for $V,W\in\lie g$, $V^{\vee}$ and $W^{\ast}$ denote the action vectors induced by $V$ and $W$ relative to the $G$-actions on $P$ and $F$, respectively; see \cite[Section 2]{Cavenaghi2022} for further clarifications.

	Fix a bi-invariant metric $Q$ on $G$. Since any two Riemannian metrics are pointwise related by a symmetric tensor, to $\ga$ and $\ga_F$ one associates symmetric tensors $O,~O_F$, named \emph{the orbit tensors associated to} $\ga,~\ga_F$, which compare $\ga$ with $Q$ and $\ga_F$ with $Q$, thus codifying the geometry of the orbits of $G$ on $P$ and on $F$: for every $U,V\in\lie g$,
$$
     \ga_F(U^*,V^*) = Q(O_FU,V), \quad
     \ga(U^{\vee},V^{\vee}) = Q(OU,V).
$$
	Both tensors are positive semi-definite, and the kernel of $O_F$ at a point $f$ is the isotropy algebra $\lie g_f$. Equivalently, $O_F$ is positive-definite on the $Q$-orthogonal complement $\lie m_f$ of $\lie g_f$ in $\lie g$, whose dimension is constant along the regular part of $F$, which we now describe. 
    
    Since the compact group $G$ acts smoothly on the connected manifold $F$, there exists an open and dense subset $F^{reg}\subset F$, the set of \emph{regular points}, any two of whose points have conjugate isotropy subgroups. We stress that the regular points are precisely those lying in principal orbits: $f\in F^{reg}$ if and only if $G_f$ is conjugate to a subgroup of $G_{f'}$ for every $f'\in F$, that is, if and only if $G_f$ has minimal isotropy type (see \cite[Section 3.5]{alexandrino2015lie}). Furthermore, $F^{reg}/G$ is a smooth manifold and, for any $G$-invariant metric $\ga_F$ on $F$, the metric induced on $F^{reg}/G$ by declaring $F^{reg}\to F^{reg}/G$ to be a Riemannian submersion is named \emph{the orbit distance metric}. For a point $f\in F\setminus F^{reg}$ the isotropy subgroup $G_f$ strictly contains a conjugate of the principal one, and whenever the orbit $Gf$ is singular, i.e. $\dim Gf<\dim Gf'$ for $f'\in F^{reg}$, the algebra $\lie g_f$ grows accordingly and the orbit tensor $O_F$ loses rank. This is explored in Lemma \ref{lem:blowingup}.

\begin{remark}\label{rem:effectivity}
Hypothesis (a) of Theorem \ref{thm:associated} forces the kernel of the $\mathrm{SU}(3)$-action on $F$ to be the center $\zeta_3$ of $\mathrm{SU}(3)$. In particular, no effective action satisfies it. Indeed, every maximal closed subgroup $H<\mathrm{SU}(3)$ contains $\zeta_3$: the product $H\zeta_3$ is a closed subgroup containing $H$, and $H\zeta_3/H$ has at most three elements while $\mathrm{SU}(3)/H$ is infinite, $H$ being proper and closed in the connected group $\mathrm{SU}(3)$. Hence, $H\zeta_3\neq\mathrm{SU}(3)$ and, by maximality, $\zeta_3\subset H$. Being central, $\zeta_3\subset gHg^{-1} = \mathrm{SU}(3)_{gf}$ for every $g\in\mathrm{SU}(3)$ and every $f$ in a principal orbit, so $\zeta_3$ fixes the dense subset $F^{reg}$ pointwise, hence all of $F$. Conversely, the kernel is a proper closed normal subgroup of $\mathrm{SU}(3)$, hence contained in $\zeta_3$ by the simplicity of $\lie{su}(3)$.

Consequently, $\mathrm{PU}(3) := \mathrm{SU}(3)/\zeta_3$ acts effectively on $F$ with principal isotropy the maximal closed subgroup $H/\zeta_3$, and, setting $\overline P := P/\zeta_3$, one has $M = \overline P\times_{\mathrm{PU}(3)}F$. That is, $M$ is a fiber bundle with structure group $\mathrm{PU}(3)$. This substitution leaves every metric constructed below unchanged: $P\to\overline P$ is a Riemannian $\zeta_3$-covering intertwining the two actions, $\lie{pu}(3) = \lie{su}(3)$, so isotropy algebras, orbit tensors and Cheeger deformations coincide, and the two quotient descriptions of $M$ induce the same submersion metrics. In particular, the results of \cite{Cavenaghi2022}, stated there under an effectivity assumption, apply to our setting after this harmless replacement. We keep $G = \mathrm{SU}(3)$ throughout, since the classification in \cite{forger} and the computations of Lemma \ref{lem:propersubgroups} are phrased in $\mathrm{SU}(3)$.
\end{remark}

It comes from \cite[Section 3.1]{Cavenaghi2022} that, fixed $(p,f)\in P\times F$, a vector $\overline X = (X+V^{\vee},X_F+W^*)$ is $\ga+\ga_F$-orthogonal to the orbit of the action \eqref{eq:action} if and only if
	\begin{equation}\label{eq:horbarpi}
	\overline{X} = \left(X-(O^{-1}O_FW)^{\vee},\,X_F + W^{\ast}\right)
	\end{equation}
	for some $W\in\lie m_f$. Indeed, for every $U\in\lie g$ one has $(\ga+\ga_F)\big(\overline X,(U^{\vee},U^{*})\big) = Q(OV+O_FW,U)$, so that $\overline X$ is horizontal if and only if $V = -O^{-1}O_FW$. Here, $O$ is invertible because $G$ acts freely on $P$, while $O_F^{-1}$ always stands for the inverse of $O_F|_{\lie m_f}$, which suffices since $(\,\cdot\,)^*$ annihilates $\lie g_f$. Fixing $(p,f)$, we abuse the notation and denote
	\begin{equation}\label{eq:conventionU*}
	d\bar\pi_{(p,f)}(X,X_F+U^*) := X+X_F+ U^*,
	\end{equation}
	so that \eqref{eq:horbarpi} identifies $T_{\bar\pi(p,f)}M$ with $T_bB\oplus \cal H^F_f\oplus\lie m_f$, where $b=\pi_P(p)$, $\cal H^F_f := (T_fGf)^{\perp_{\ga_F}}$, and $\lie m_f$ is the $Q$-orthogonal complement of $\lie g_f$. Notice that this decomposition holds also at points of $F\setminus F^{reg}$ with the remark that $\lie m_f$ has possibly smaller dimension.

	From now on, $\ga$ is a $G$-invariant metric \emph{on $P$} and $\ga_F$ a $G$-invariant metric \emph{on $F$}, chosen so that $\bar\pi\colon (P\times F,\ga+\ga_F)\to (M,\aga)$ is a Riemannian submersion. We write $\ga_t$ for the Cheeger deformation of $\ga$ \emph{with respect to the $G$-action on $P$}\footnote{See for instance \cite{cheeger,mutterz,Cavenaghi2023}.}, and $\aga_t$ for the Riemannian submersion metric on $M$ induced by $\ga_t+\ga_F$. Moreover, $\ga_B$ denotes the metric on $B$ for which $\pi_P\colon (P,\ga)\to (B,\ga_B)$ is a Riemannian submersion. Finally, following \cite[Section 3.1]{Cavenaghi2022}, two-planes must be reparametrized before deforming: on $\lie m_f$ one sets $\widetilde O_t := (O_F^{-1}+O_t^{-1})^{-1}$ and $\widetilde C_t := -O^{-1}\widetilde O_t$, extends $\widetilde C_t$ to $T_{\bar\pi(p,f)}M$ by $\widetilde C_t(X+X_F+U^*) := X+X_F+(\widetilde C_tU)^*$, and puts
	$$\tilde\kappa_{\aga_t}(\widetilde X,\widetilde Y) := R_{\aga_t}\big(\widetilde C_t^{-1}\widetilde X,\widetilde C_t^{-1}\widetilde Y,\widetilde C_t^{-1}\widetilde Y,\widetilde C_t^{-1}\widetilde X\big).$$
	With these conventions, \cite[Theorem 3.1]{Cavenaghi2022} reads:
	\begin{theorem}\label{thm:secnew}
		Let $\ga$, $\ga_F$, $\ga_t$ and $\aga_t$ be as above. Then for every pair $\tilde X = X + X_F + U^*$, $\tilde Y = Y + Y_F + V^*$ in $T_{\bar\pi(p,f)}M$,
		\begin{equation}\label{eq:secnew}
		\tilde{\kappa}_{\aga_t}(\tilde X, \tilde Y) =  \kappa_{\ga_t}(X{+}U^{\vee},Y{+}V^{\vee}) {+} K_{\ga_F}\big(X_F - (O_F^{-1}OU)^*, Y_F - (O_F^{-1}OV)^*\big) + \tilde z_t(\tilde X,\tilde Y),
		\end{equation}
		where $\kappa_{\ga_t}(\overline X,\overline Y) := R_{\ga_t}(C_t^{-1}\overline X,C_t^{-1}\overline Y,C_t^{-1}\overline Y,C_t^{-1}\overline X)$ is the unreduced sectional curvature of $\ga_t$ evaluated on the reparametrized two-plane $C_t^{-1}(X{+}U^{\vee})\wedge C_t^{-1}(Y{+}V^{\vee})$ of $TP$, $C_t$ being the metric tensor of $\ga_t$, and $\tilde z_t\geq 0$.
	\end{theorem}

	We now specialize to $G = \mathrm{SU}(3)$ and $B = W_6 = \mathrm{SU}(3)/\mathrm T^2$. Fix $f\in F$, let $\{e^F_1,\ldots,e^F_{\dim \cal H^F_f}\}$ be a $\ga_F$-orthonormal basis of $\cal H^F_f$, let $\{v_1(0),\ldots,v_{\dim\lie m_f}(0)\}$ be a $Q$-orthonormal basis of $\lie m_f$ (so that $\{e^F_j\}\cup\{v_k(0)^*\}$ is a basis of $T_fF$) and let $\{e^B_i\}$ be a $\ga$-orthonormal basis of the $\ga$-orthogonal complement of the $G$-orbit at $p$. For $0\le n\le \dim\cal H^F_f$ we write
	\begin{equation}\label{eq:hRicci}
	\Ricci^{\mathbf h}_{n,\ga_F}(X_F) := \sum_{j=1}^{n}K_{\ga_F}(X_F,e^F_j)
	\end{equation}
	for the \emph{horizontal partial Ricci curvature} of $(F,\ga_F)$ at $f$, that is, the sum of $n$ sectional curvatures of $\ga_F$ determined by $X_F$ and directions normal to the orbit $\mathrm{SU}(3)f$ (compare with the horizontal Ricci curvature in \cite[Section 2.1]{Cavenaghi2022}). Adapting the proof of Lemmas 5 and 7 in \cite{Cavenaghi2022} we obtain:
	\begin{lemma}\label{thm:limitintermediate}
		Let $M = P\times_{\mathrm{SU}(3)}F\to B = \mathrm{SU}(3)/\mathrm T^2$ be as above and let $\aga_t$ be as in Theorem \ref{thm:secnew}. Fix $(p,f)\in P\times F$ and $X+X_F+U^*\in T_{\bar\pi(p,f)}M$ as in \eqref{eq:conventionU*}. Choose an integer $1\leq d\leq \dim M -1$ and non-negative integers $n_B,~n_{\cal H^F_f},~n_{\lie m_f}$ constrained by
		$$d = n_B+n_{\cal H^F_f}+n_{\lie m_f},\qquad n_B\leq 6,\quad n_{\cal H^F_f}\leq \dim\cal H^F_f,\quad n_{\lie m_f}\leq \dim\lie m_f.$$
		Then, computing the intermediate Ricci curvatures on the left-hand side with respect to $\aga_t$-orthonormal $d$-frames orthogonal to $X+X_F+U^*$ and adapted to the splitting $T_{\bar\pi(p,f)}M\cong T_bB\oplus\cal H^F_f\oplus\lie m_f$, we have
		\begin{equation}\label{eq:asymptotic}
		 	\liminf_{t\to \infty}\Ricci_{d,\aga_t}(X + X_F + U^*) \geq \Ricci_{n_B,\ga_B}(d\pi_P X) + \Ricci^{\mathbf{h}}_{n_{\cal H_f^F},\ga_F}(X_F) + \sum_{j=1}^{n_{\cal H_f^F}}\frac{3}{4}\big|[X_F,e^{F}_j]^{\lie m_f}\big|_{\ga_F}^2 + \sum_{k=1}^{n_{\lie m_f}}\frac{1}{4}\big\|[v_k(0),U]\big\|_Q^2,
		\end{equation}
		where the $v_k(0)$ appearing in the last sum may be taken $Q$-orthogonal to $U$.
 	\end{lemma}
	The proof runs as the one of \cite[Lemmas 5 and 7]{Cavenaghi2022}, which are stated at an arbitrary point $(p,f)$ and where the case of the full Ricci curvature is treated: one applies Theorem \ref{thm:secnew} to the pairs formed by $\tilde X$ and the elements of the $\aga_t$-orthonormal basis of the $\bar\pi$-horizontal space constructed there, discards $\tilde z_t\geq 0$, and sums over the $d$ chosen elements of that basis instead of over all of them, the asymptotics of $\widetilde O_t$ used there being unaffected. The final clause follows since $\aga_t$ restricted to the $\lie m_f$-directions is asymptotic to a multiple of $Q$, so that $\aga_t$-orthogonality of the frame to $\tilde X$ yields, in the limit, $Q$-orthogonality to $U$. The only genuine adaptation concerns the first and third terms of \eqref{eq:asymptotic}: since $\|[X,Y]^{\cal V}\|_{\ga_t}\to 0$ as $t\to\infty$ for $\pi_P$-horizontal $X,Y$, the Gray--O'Neill formula applied to $\pi_P\colon (P,\ga_t)\to (B,\ga_B)$ gives $\kappa_{\ga_t}(X,Y)\to \kappa_{\ga_B}(d\pi_PX,d\pi_PY)$, and the analogous argument along the fibers accounts for the $\tfrac34$-terms (compare with Lemmas 2.6 and 4.2 in \cite{Cavenaghi2023}).

Let us fix the following conventions for the rest of this section: for $f\in F^{reg}$ we abbreviate $m := \dim\mathrm{SU}(3)f = \dim\lie m_f$ and $\ell := \dim\cal H^F_f = \dim F - m$, both independent of $f\in F^{reg}$, and recall from the statement of Theorem \ref{thm:associated} the constants
\begin{equation}\label{eq:k0}
k_2 := \max\{d_F+m,\ \ell+\delta\},\qquad k_0 := k_2+6.
\end{equation}

\begin{lemma}\label{lem:propersubgroups}
    Assume that a principal orbit of the $\mathrm{SU}(3)$-action on $F$ has as isotropy subgroup a maximal closed subgroup $H<\mathrm{SU}(3)$. Then:
    \begin{enumerate}[(1)]
    \item every non-principal orbit is a fixed point, that is, $\mathrm{SU}(3)f = f$ for every $f\in F\setminus F^{reg}$;
    \item for every $f'\in F^{reg}$ the orbit $\mathrm{SU}(3)f'$ is, up to a finite Riemannian covering, one of
    $$\mathrm{SU}(3)/\mathrm{U}(2),\qquad \mathrm{SU}(3)/\mathrm{SO}(3),\qquad \mathrm{SU}(3)/\mathrm{T}^2,\qquad \mathrm{SU}(3),$$
    so that $m = \dim\mathrm{SU}(3)f'\in\{4,5,6,8\}$;
    \item endowed with the normal homogeneous metric induced by a bi-invariant metric $Q$ on $\mathrm{SU}(3)$, the orbit $\mathrm{SU}(3)f'$ has positive $\Ricci_{\delta}$ and not positive $\Ricci_{\delta-1}$, where $\delta = \delta(H) = \max_{x\in\lie m\setminus\{0\}}\dim  \{y \in \lie m : [x,y]=0\}$ equals
    $$\delta = 1,\ 3,\ 2,\ 4$$
    in the four cases of (2), respectively. In particular, ${m-4\leq \delta\leq m-2}$ in all cases.
    \end{enumerate}
\end{lemma}
\begin{proof}
    \emph{(1)} Let $f\in F$ and put $\overline H := \mathrm{SU}(3)_f$. Since the principal isotropy groups are, up to conjugation, contained in every isotropy group (see \cite[Section 3.5]{alexandrino2015lie}), we may assume $H<\overline H$ after conjugating. As $H$ is maximal, either $\overline H = H$, in which case $\mathrm{SU}(3)f$ has the same orbit type as a principal orbit and $f\in F^{reg}$, or $\overline H = \mathrm{SU}(3)$, in which case $\mathrm{SU}(3)f = f$.

    \emph{(2)} If $f'\in F^{reg}$ then $\mathrm{SU}(3)f'\cong\mathrm{SU}(3)/H$ with $H$ a maximal proper closed subgroup of $\mathrm{SU}(3)$. According to \cite[Section 8.1, p. 1006]{forger}, the possibilities for $H$ are:
    \begin{enumerate}[(i)]
        \item Type 1: \textbf{normalizer of a maximal connected subgroup},
        \begin{align*}
& \mathrm{U}(2) \cong \mathrm{S}(\mathrm{U}(2)\times \mathrm{U}(1)),\\
& \zeta_3 \times \mathrm{SO}(3);
\end{align*}
\item Type 2: \textbf{finite maximal closed subgroup},
\begin{align*}
& \zeta_3 \times \mathrm{GL}(3,2) \equiv \zeta_3\times \mathrm{Gl}(3,\bb F_2),\\
& 3.\mathrm{A}_6,\\
& \mathrm{S}(\mathrm{Cl}_1(3));
\end{align*}
\item Type 3: \textbf{normalizer of a positive dimensional non-maximal connected subgroup},
\[
\mathrm{S}(\mathrm{U}(1)\times \mathrm{U}(1) \times \mathrm{U}(1))\rtimes \mathrm{S}_3,
\]
    \end{enumerate}
    where $\mathrm{S}(\mathrm{Cl}_1(3))$ is the determinant minus one subgroup of the single qutrit Clifford group, i.e., Shephard--Todd-25, \cite{shephard_todd_1954}, and $3.\mathrm{A}_6$ is the Valentiner group \cite{valentiner_1889}; $\zeta_3$ denotes the center of $\mathrm{SU}(3)$, contained in each of the groups above in accordance with Remark \ref{rem:effectivity}. The identity components of these groups are $\mathrm{U}(2),~\mathrm{SO}(3),~\{e\}$ and $\mathrm{T}^2$, with Lie algebras $\lie u(2),~\lie{so}(3),~0$ and $\lie t^2$. Since $\mathrm{SU}(3)/H_0\to\mathrm{SU}(3)/H$ is a finite Riemannian covering for the normal homogeneous metrics, claim (2) follows, together with $m\in\{4,5,6,8\}$.

    \emph{(3)} By Proposition 3.1 in \cite{DGM}, if $\lie{su}(3) = \lie m\oplus\lie h$ is the standard reductive decomposition of $\mathrm{SU}(3)/H$, then the minimum value of $d$ yielding positive $\Ricci_{d,\mathrm{SU}(3)/H}$ at the normal homogeneous metric equals $\delta$, a quantity which only depends on $\lie h$, hence is unchanged by the coverings of (2). We compute it in the four cases. Throughout, recall that for the normal homogeneous metric one has $K(x,y) = \frac14\|[x,y]_{\lie m}\|^2_Q + \|[x,y]_{\lie h}\|^2_Q$ for $Q$-orthonormal $x,y\in\lie m$, so that $K(x,y) = 0$ if and only if $[x,y]=0$. Also, recall that the centralizer of a non-zero element of $\lie{su}(3)$ has dimension $2$ if the element is regular; and dimension $4$, in which case it is isomorphic to $\lie u(2)$, if it is singular.
    \begin{itemize}
    \item $\lie h = \lie u(2)$: here $\mathrm{SU}(3)/\mathrm{U}(2)\cong\CP^2$ is isotropy irreducible and the normal metric is the Fubini--Study one, which has positive sectional curvature. Hence $\delta = 1$.
    \item $\lie h = \lie{so}(3)$: here $\lie m = \{\mathbf i S: S \text{ real symmetric traceless}\}$ and, for the singular element $x = \mathbf{i}\operatorname{diag}(1,1,-2)\in\lie m$, the condition $[x,\mathbf iS]=0$ amounts to $S_{13}=S_{23}=0$, which leaves $\dim  \{y \in \lie m : [x,y]=0\} = 3$. As $\dim \{y \in \lie m : [x,y]=0\} = 2$ for regular $x$, we have that $\delta = 3$.
    \item $\lie h = \lie t^2$: for regular $x$ one has $\dim  \{y \in \lie m : [x,y]=0\}\leq\dim  \{y \in \lie{su}(3) : [x,y]=0\} = 2$. For singular $x$, take for instance
    \[x = \begin{pmatrix}0&1&1\\-1&0&\mathbf i\\-1&\mathbf i&0\end{pmatrix}\in\lie m,\qquad \text{with eigenvalues } -\mathbf i,-\mathbf i,2\mathbf i,\]
    whose eigenspace decomposition is $\bb C^3 = E\oplus L$, $L = \bb C(1,\mathbf i,\mathbf i)$. Writing $ \{y \in \lie{su}(3) : [x,y]=0\} = \bb Rx\oplus\lie{su}(E)$ and imposing that the diagonal entries vanish, a direct computation in the orthonormal basis $p = \frac{1}{\sqrt2}(0,1,-1)$, $q = \frac{1}{\sqrt6}(2\mathbf i,1,1)$ of $E$ shows that $\lie{su}(E)\cap\lie m$ is one dimensional, so that $\dim  \{y \in \lie{m} : [x,y]=0\} = 2$ and $\delta = 2$.
    \item $\lie h = 0$: here $\lie m = \lie{su}(3)$ and $\delta = \max_x\dim  \{y \in \lie{su}(3) : [x,y]=0\} = 4$, attained at any singular element.\qedhere
    \end{itemize}
\end{proof}

\begin{lemma}\label{lem:blowingup}
Let $F$ be as in the hypotheses of Theorem \ref{thm:associated}, let $\ga_F$ be an $\mathrm{SU}(3)$-invariant metric on $F$ and let $f\in F\setminus F^{reg}$, so that $f$ is a fixed point of the action by Lemma \ref{lem:propersubgroups}(1). Denote by $(\ga_F)_s$ the Cheeger deformation of $\ga_F$ with respect to the $\mathrm{SU}(3)$-action and by $\cal F\subset F$ the connected component through $f$ of the fixed point set. Then, for every $C>0$ and every $n$ with $\dim\cal F\leq n\leq\dim F-1$, there exists $s = s(C,n)>0$ such that
\[\Ricci_{n,(\ga_F)_s}(W) > C\]
for every $(\ga_F)_s$-unit vector $W\in T_fF$. Moreover, $(\ga_F)_s$ is $\mathrm{SU}(3)$-invariant and induces on $F^{reg}/\mathrm{SU}(3)$ the same orbit distance metric as $\ga_F$.
\end{lemma}
\begin{proof}
Since $f$ is a fixed point, $\lie{su}(3)_f = \lie{su}(3)$ and the orbit tensor $O_F$ vanishes at $f$. In particular, $C_s = \operatorname{id}$ on $T_fF$, so that $(\ga_F)_s$ and $\ga_F$ agree as bilinear forms at $f$ and only the curvature is affected by the deformation.  Proposition 1.3 in \cite{mutterz} says that for $\overline X = X+U^*$ and $\overline Y = Y+V^*$ one has
\[\kappa_{s}(\overline X,\overline Y) = \kappa_{0}(\overline X,\overline Y)+\frac{s^3}{4}\|[O_FU,O_FV]\|_Q^2+z_s(\overline X,\overline Y),\qquad z_s\geq 0\ \text{non-decreasing in } s.\]
By Lemma 3.5 and Proposition 3.4 in \cite{Cavenaghi2023}, $z_s(\overline X,\overline Y)\nearrow +\infty$ as $s\nearrow+\infty$ whenever the plane $\overline X\wedge\overline Y$ contains a direction transverse to $\cal F$, the divergence being uniform on the compact set of such planes at $f$.

Let now $W\in T_fF$ be $(\ga_F)_s$-unit and let $\{e_1,\ldots,e_n\}$ be orthonormal and orthogonal to $W$. Since $n\geq\dim\cal F$, the $n+1$ orthonormal vectors $W,e_1,\ldots,e_n$ cannot all be tangent to $\cal F$, so that
\[\Ricci_{n,(\ga_F)_s}(W) = \sum_{j=1}^n\kappa_0(W,e_j)+\sum_{j=1}^nz_s(W,e_j)\geq -n\max_{F}|\kappa_0| +\max_j z_s(W,e_j),\]
and the right-hand side is bigger than $C$ for $s$ large, uniformly in $W$ and in the frame, by compactness. Finally, a Cheeger deformation is by construction $\mathrm{SU}(3)$-invariant and leaves the induced metric on the orbit space unchanged.\qedhere
\end{proof}

 By Lemma \ref{lem:propersubgroups} one has $m\geq 4$, $\delta\leq m-1$ and $m-5\leq\delta$, hence $\dim F = m+\ell\geq 6$. Using that $d_F\leq\ell-1$ we get moreover that $\dim F+2\leq k_0\leq\dim M-1$.
 
\begin{proposition}\label{prop:firstpart}
    Let $M = P\times_{\mathrm{SU}(3)}F \to \mathrm{SU}(3)/\mathrm{T}^2$ be as in the hypotheses of Theorem \ref{thm:associated}. Then, for any invariant metric $\ga_B$ on $B = \mathrm{SU}(3)/\mathrm{T}^2$ with positive sectional curvature and any $k\ge k_0$, after choosing a Riemannian metric on $F$ appropriately, the total space $M$ carries a Riemannian submersion metric with positive $\Ricci_{k}$.
\end{proposition}
\begin{proof}
	Let $\ga_B$ be an invariant metric on $B$ with positive sectional curvature and let $\ga_F$ be an $\mathrm{SU}(3)$-invariant metric on $F$ whose orbit distance metric $\overline\ga_F$ on $F^{reg}/\mathrm{SU}(3)$ satisfies $\Ricci_{d_F,\overline\ga_F}>0$, as granted by hypothesis (b). Since the fixed point set is finite, Lemma \ref{lem:blowingup} allows us to replace $\ga_F$, once and for all, by a finite Cheeger deformation of itself such that
	\begin{equation}\label{eq:fixedpointprep}
	\Ricci_{n,\ga_F}(W)>0\quad\text{for every fixed point } f,\ \text{every unit } W\in T_fF\ \text{and every } 1\leq n\leq\dim F-1,
	\end{equation}
	taking the largest of the finitely many deformation parameters involved, which is possible since the deformed curvatures are non-decreasing in the parameter. We regard $M$ with the connection metric $\aga$ furnished by \cite[Proposition 2.7.1]{gw}, for which $\overline\pi$ is a Riemannian submersion and the fibers $F$ are totally geodesic, and with the deformed family $\aga_t$ of Theorem \ref{thm:secnew}, so that $\aga_0 = \aga$.

	Fix $(p,f)\in P\times F$ and recall from \eqref{eq:conventionU*} the splitting $T_{\overline\pi(p,f)}M\cong T_bB\oplus\cal H^F_f\oplus\lie m_f$. Since the right-hand side of \eqref{eq:asymptotic} preserves this splitting, in bounding $\Ricci_{k,\aga_t}$ from below it suffices to consider $k$-frames adapted to it, that is, to let
	\begin{equation}\label{eq:splitting}
	k = n_B+n_{\lie m_f}+n_{\cal H^F_f}
	\end{equation}
	run over all admissible decompositions, subject to
	$$0\le n_B\le 6 = \dim\mathrm{SU}(3)/\mathrm T^2,\qquad 0\le n_{\lie m_f}\le\dim\lie m_f\le 8 = \dim \mathrm{SU}(3),\qquad 0\le n_{\cal H^F_f}\le \dim\cal H^F_f,$$
	the bound on $n_{\lie m_f}$ following from $\lie m_f\subset\lie{su}(3)$.

	Assume, by contradiction, that for every $t>0$ there are $q_t\in M$ and an $\aga_t$-unit vector $\widetilde X_t\in T_{q_t}M$ with $\Ricci_{k,\aga_t}(\widetilde X_t)\le 0$. Taking $t = n\in\bb N$, using the compactness of $P\times F$ and arguing as in the proof of \cite[Theorem 1.1]{Cavenaghi2022}, we obtain a point $(p,f)\in P\times F$, a decomposition \eqref{eq:splitting} and a non-zero $\widetilde X = X+X_F+U^*\in T_{\overline\pi(p,f)}M$ for which the right-hand side of \eqref{eq:asymptotic} is non-positive:
	\begin{equation}\label{ineq}
	\Ricci_{n_B,\ga_B}(d\pi_PX) + \Ricci^{\mathbf h}_{n_{\cal H^F_f},\ga_F}(X_F) + \sum_{j=1}^{n_{\cal H^F_f}}\frac34\big|[X_F,e^F_j]^{\lie m_f}\big|^2_{\ga_F} + \sum_{k=1}^{n_{\lie m_f}}\frac14\big\|[v_k(0),U]\big\|^2_Q\le 0,
	\end{equation}
	the $v_k(0)$ being $Q$-orthogonal to $U$. We show that \eqref{ineq} is impossible.

	\emph{Step 1.} Suppose $f\in F\setminus F^{reg}$. By Lemma \ref{lem:propersubgroups}(1), $f$ is a fixed point, so $\lie m_f = 0$, $U^* = 0$, $n_{\lie m_f} = 0$ and $\cal H^F_f = T_fF$; the last two summands of \eqref{ineq} vanish and $\Ricci^{\mathbf h}$ is the ordinary partial Ricci curvature of $\ga_F$ at $f$. If $X_F\neq 0$, then $n_{\cal H^F_f} = k-n_B\geq k-6\geq\dim F-4\geq 1$, so the second summand of \eqref{ineq} is positive by \eqref{eq:fixedpointprep} and the first is non-negative, a contradiction. If $X_F = 0$, then $X\neq 0$ and $n_B\geq k-\dim F\geq 2$, so the first summand is positive, again a contradiction. Hence $f\in F^{reg}$.

	\emph{Step 2.} As $f\in F^{reg}$, the Gray--O'Neill formula for the Riemannian submersion $F^{reg}\to F^{reg}/\mathrm{SU}(3)$ gives
	\begin{equation}\label{eq:oneillF}
	\Ricci_{n_{\cal H^F_f},\overline\ga_F}\big(\overline{X_F}\big) = \Ricci^{\mathbf h}_{n_{\cal H^F_f},\ga_F}(X_F) + \sum_{j=1}^{n_{\cal H^F_f}}\frac34\big|[X_F,e^F_j]^{\lie m_f}\big|^2_{\ga_F},
	\end{equation}
	$\overline{X_F}$ denoting the projection of $X_F$, so that \eqref{ineq} reads
	\begin{equation}\label{ineq1}
	\Ricci_{n_B,\ga_B}\big(d\pi_PX\big) + \Ricci_{n_{\cal H^F_f},\overline\ga_F}\big(\overline{X_F}\big) + \sum_{k=1}^{n_{\lie m_f}}\frac14\big\|[v_k(0),U]\big\|^2_Q\le 0.
	\end{equation}
	We record the sign of each summand. The first is non-negative, and positive if $n_B\geq1$ and $X\neq0$, since $\ga_B$ has positive sectional curvature. The second is non-negative, and positive if $n_{\cal H^F_f}\geq d_F$ and $X_F\neq 0$, since $\Ricci_{d_F,\overline\ga_F}>0$ implies $\Ricci_{n,\overline\ga_F}>0$ for every $n\geq d_F$. The third is non-negative, and positive if $n_{\lie m_f}\geq\delta$ and $U\neq0$: indeed, the $v_k(0)$ are $Q$-orthonormal and $Q$-orthogonal to $U$, whereas, by Lemma \ref{lem:propersubgroups}(3), $\dim{\{y \in \lie m_f : [U,y]=0\}}\leq\delta$. Since $U\in {\{y \in \lie m_f : [U,y]=0\}}$, at most $\delta-1$ mutually $Q$-orthonormal vectors orthogonal to $U$ can commute with $U$, so at least one summand $\frac14\|[v_k(0),U]\|^2_Q$ is positive.

	\emph{Step 3.} Since $\widetilde X\neq0$, at least one of $X,~X_F,~U$ is non-zero.
	\begin{itemize}
	\item If $X_F\neq0$, then $n_{\cal H^F_f}\geq k - 6 - \dim\lie m_f = k-6-m\geq d_F$, because $k\geq k_0\geq d_F+m+6$; the second summand of \eqref{ineq1} is positive, contradicting \eqref{ineq1}.
	\item If $X\neq0$, then $n_B\geq k-\dim\cal H^F_f-\dim\lie m_f = k-\dim F\geq 2$, because $k\geq k_0\geq\dim F+2$; the first summand is positive, contradicting \eqref{ineq1}.
	\item If $U\neq0$, then $n_{\lie m_f}\geq k-6-\dim\cal H^F_f = k-6-\ell\geq \delta$, because $k\geq k_0\geq\ell+\delta+6$; the third summand is positive, contradicting \eqref{ineq1}.
	\end{itemize}
	In all cases we contradict \eqref{ineq1}. Thus, there is $t>0$ with $\Ricci_{k,\aga_t}>0$ on $M$. \qedhere
\end{proof}

\subsection{The proof of Theorem \ref{thm:associated}}
Let $\ga_F$ be an $\mathrm{SU}(3)$-invariant metric on $F$ as in hypothesis (b), and replace it by a finite Cheeger deformation of itself, which changes neither its invariance nor the orbit distance metric it induces on $F^{reg}/\mathrm{SU}(3)$. Arguing exactly as in Steps 1--3 of the proof of Proposition \ref{prop:firstpart}, but with the classical Cheeger deformation of $(F,\ga_F)$ in place of $\aga_t$ such a deformation can be chosen so that
	\begin{equation}\label{eq:fiberindex}
	\Ricci_{k_2,\ga_F}>0\ \text{ on } F,\qquad k_2 = \max\{d_F+m,~\ell+\delta\},
	\end{equation}
	the two entries of the maximum accounting, respectively, for the vectors with non-trivial $\cal H^F_f$-component and for those tangent to the orbits, and Lemma \ref{lem:blowingup} accounting for the finitely many fixed points, as in \eqref{eq:fixedpointprep}. Next, let $\ga_B(t)$ solve the projected homogeneous Ricci flow with $\ga_B(0) = \ga_B$, and let $\aga^{(t)}$ be the connection metric on $M$ determined by $\ga_B(t)$, by $\ga_F$ and by a fixed principal connection on $P$, as in Proposition 2.7.1 in \cite{gw}; its fibers are totally geodesic and isometric to $(F,\ga_F)$. Finally, for $s\in\bb R$ we let $\aga^{(t)}_{s}$ denote the $s$-canonical variation of $\aga^{(t)}$, namely
	\[\aga^{(t)}_{s} := \aga^{(t)}\big|_{\cal H\times\cal H}+e^{2s}\,\aga^{(t)}\big|_{\cal V\times\cal V},\]
	$\cal H$ and $\cal V$ being the horizontal and the vertical distributions of $M\to B$. The family in the statement of Theorem \ref{thm:associated} is $\ga_t := \aga^{(t)}_{s_0}$, for a single $s_0<0$ chosen below.

Since $\Rsec{n_B}$ is open and $t\mapsto\ga_B(t)$ is continuous, there is $\tau>0$ with $\ga_B(t)\in\Rsec{n_B}$ for $|t|\leq\tau$. By \eqref{eq:fiberindex} and by $k_0 = n_B+\dim F = k_2+6$, Theorem A in \cite{reiser2022positive}, applied to the Riemannian submersion $(M,\aga^{(t)})\to (B,\ga_B(t))$ with totally geodesic fibers, base with $\Ricci_{n_B}>0$ and fibers with $\Ricci_{k_2}>0$, yields for each $t\in[-\tau,\tau]$ a number $s(t)<0$ such that $\aga^{(t)}_{s}\in\Rsec{k}$ for every $k\geq \max\{n_B+\dim F,\ k_2+6\} = k_0$ and every $s<s(t)$. Compactness of $[-\tau,\tau]$ and continuity of the curvature in $t$ allow us to take $s(t)$ independent of $t\in[-\tau,\tau]$.

	 We claim that if, for some $t$, there are $\ga_B(t)$-orthonormal vectors $X,Y_1,\ldots,Y_{n_B}\in T_o\mathrm{SU}(3)/\mathrm T^2$ with
	\begin{equation}\label{eq:basenegative}
	\Ricci_{n_B,\ga_B(t)}(X) = \sum_{j=1}^{n_B}K_{\ga_B(t)}(X,Y_j) = -2\eta<0,
	\end{equation}
	then $\aga^{(t)}_{s}\notin\Rsec{k_0}$ for every $s$ sufficiently negative. Indeed, fix $q\in M$ over $o$ and denote by $\widehat X,\widehat Y_j\in\cal H_q$ the horizontal lifts, which remain orthonormal for $\aga^{(t)}_{s}$ because the canonical variation does not alter $\aga^{(t)}\big|_{\cal H\times\cal H}$. Let $\{U_1,\ldots,U_{\dim F}\}$ be an $\aga^{(t)}$-orthonormal basis of $\cal V_q$, so that $\{{e^{-s}}U_i\}$ is $\aga^{(t)}_{s}$-orthonormal. Then
	\[\big\{\widehat Y_1,\ldots,\widehat Y_{n_B},\,{e^{-s}}U_1,\ldots,{e^{-s}}U_{\dim F}\big\}\]
	is $\aga^{(t)}_{s}$-orthonormal, orthogonal to $\widehat X$, and has $n_B+\dim F = k_0$ elements. Denoting by $A$ the integrability tensor of $M\to B$ computed for $\aga^{(t)}$, the Gray--O'Neill formulae for a submersion with totally geodesic fibers under the canonical variation give
	\[K_{\aga^{(t)}_{s}}\big(\widehat X,\widehat Y_j\big) = K_{\ga_B(t)}(X,Y_j)-3e^{2s}\big|A_{\widehat X}\widehat Y_j\big|^2,\qquad K_{\aga^{(t)}_{s}}\big(\widehat X,{e^{-s}}U_i\big) = e^{2s}\big|A_{\widehat X}U_i\big|^2.\]
	Since $\lr{A_{\widehat X}U,\widehat Z} = -\lr{U,A_{\widehat X}\widehat Z}$ for $U$ vertical and $\widehat Z$ horizontal, one has $\sum_{i}|A_{\widehat X}U_i|^2 = \sum_{j=1}^{6}|A_{\widehat X}\widehat e_j|^2 =: \|A_{\widehat X}\|^2$ for any orthonormal basis $\{\widehat e_j\}$ of $\cal H_q$. Summing, and discarding the non-positive terms,
	\begin{equation}\label{eq:escaping}
	\Ricci_{k_0,\aga^{(t)}_{s}}\big(\widehat X\big)\leq \Ricci_{n_B,\ga_B(t)}(X)+e^{2s}\big\|A_{\widehat X}\big\|^2 = -2\eta+e^{2s}\big\|A_{\widehat X}\big\|^2,
	\end{equation}
	which is $\leq-\eta<0$ as soon as $e^{2s}\|A_{\widehat X}\|^2\leq\eta$, proving the claim. 

	 Assume $n_B\neq4$ and let $\ga_B$ be as in Theorem \ref{ithm:main}(a) for $d = n_B$, as specified in the statement, so that the projected Ricci flow leaves $\Rsec{n_B}$; as the trajectory then converges to one of the collapsed metrics $L$, $M$, $N$ (for $n_B\in\{2,3\}$), or emanates from one of the vertices $O$, $P$, $Q$ (for $n_B = 5$), it meets the open set of metrics for which \eqref{eq:basenegative} holds (see the panel $d = n_B$ of Figure \ref{fig:sec-ric}). Pick $t^*$ in that open set, together with $\eta>0$ and a neighborhood $I$ of $t^*$ on which \eqref{eq:basenegative} holds with the same $X,Y_j$ and the same $\eta$. Choosing $s_0<0$ negative enough for both the second and the third steps, and enlarging $\tau$ so that $[-\tau,\tau]\supset I$ in the second step is not needed the family $\ga_t = \aga^{(t)}_{s_0}$ satisfies $\ga_t\in\Rsec{k}$ for all $k\geq k_0$ and $|t|\leq\tau$, and $\ga_t\notin\Rsec{k_0}$ for every $t\in I$. \qed

\section*{Acknowledgements}

The São Paulo Research Foundation (FAPESP) partially supports L. F. C. grant 2022/09603-9, partially supports L. G. grant 2023/13131-8 and partially supports D. D. N. grants 2024/15612-6 and 2026/03312-3. 

The National Council for Scientific and Technological Development (CNPq) partially supports L. G. grant 306021/2024-2, partially supports R. M. M. grants  306287/2024-2 and 401974/2025-1, and partially supports D. D. N. grant 301878/2025-0.

The Coordination for the Improvement of Higher Education Personnel (CAPES), Brazil, through the MATH-AmSud program, partially supports D. D. N. grant 88881.179491/2025-01. 

L. F. C. is also supported by the Simons Foundation, grant SFI-MPS-T-Institutes-00007697, and by the Ministry of Education and Science of the Republic of Bulgaria, grant DO1-239/10.12.2024.

%\iffalse
%Teorema para colocar em algum paper, não necessariamente neste:

%\begin{theorem}
%Let $d$ be a positive integer less than or equal to the manifold's dimension. Assume that there exist $v=(v_1,\ldots,v_n)\in\mathbb{N}^n$, such that $v_1+\cdots+v_n=d$, and $x_0\in\R^n$ with positive coordinates such that 
%\[
%\langle \Ricci(x_0),v\rangle=0\,\,\text{ and }\,\,\langle (d \Ricci(x_0))^T \Ricci(x_0),v\rangle\neq0.
%\]
%$t\mapsto\langle \Ricci(F_t(x_0)),v\rangle$ changes its sign for $t$ near $0$.
%\end{theorem}
%\fi

	\bibliographystyle{alpha}

	\bibliography{main}

\end{document}